\documentclass[12pt]{amsart}

\usepackage{amsthm}
\usepackage{amsmath}
\usepackage{amssymb}
\usepackage{latexsym}
\usepackage{verbatim}
\usepackage[mathscr]{eucal}
\usepackage{oldgerm}
\usepackage{amscd}
\usepackage[all]{xy}
\begin{document}

%Greek letters

\newcommand{\alp}{\alpha}
\newcommand{\bet}{\beta}
\newcommand{\gam}{\gamma}
\newcommand{\del}{\delta}
\newcommand{\eps}{\epsilon}
\newcommand{\zet}{\zeta}
\newcommand{\tht}{\theta}
\newcommand{\iot}{\iota}
\newcommand{\kap}{\kappa}
\newcommand{\lam}{\lambda}
\newcommand{\sig}{\sigma}
\newcommand{\ups}{\upsilon}
\newcommand{\ome}{\omega}
\newcommand{\vep}{\varepsilon}
\newcommand{\vth}{\vartheta}
\newcommand{\vpi}{\varpi}
\newcommand{\vrh}{\varrho}
\newcommand{\vsi}{\varsigma}
\newcommand{\vph}{\varphi}
\newcommand{\Gam}{\Gamma}
\newcommand{\Del}{\Delta}
\newcommand{\Tht}{\Theta}
\newcommand{\Lam}{\Lambda}
\newcommand{\Sig}{\Sigma}
\newcommand{\Ups}{\Upsilon}
\newcommand{\Ome}{\Omega}

%fraktur letters

\newcommand{\frka}{{\mathfrak a}}    \newcommand{\frkA}{{\mathfrak A}}
\newcommand{\frkb}{{\mathfrak b}}    \newcommand{\frkB}{{\mathfrak B}}
\newcommand{\frkc}{{\mathfrak c}}    \newcommand{\frkC}{{\mathfrak C}}
\newcommand{\frkd}{{\mathfrak d}}    \newcommand{\frkD}{{\mathfrak D}}
\newcommand{\frke}{{\mathfrak e}}    \newcommand{\frkE}{{\mathfrak E}}
\newcommand{\frkf}{{\mathfrak f}}    \newcommand{\frkF}{{\mathfrak F}}
\newcommand{\frkg}{{\mathfrak g}}    \newcommand{\frkG}{{\mathfrak G}}
\newcommand{\frkh}{{\mathfrak h}}    \newcommand{\frkH}{{\mathfrak H}}
\newcommand{\frki}{{\mathfrak i}}    \newcommand{\frkI}{{\mathfrak I}}
\newcommand{\frkj}{{\mathfrak j}}    \newcommand{\frkJ}{{\mathfrak J}}
\newcommand{\frkk}{{\mathfrak k}}    \newcommand{\frkK}{{\mathfrak K}}
\newcommand{\frkl}{{\mathfrak l}}    \newcommand{\frkL}{{\mathfrak L}}
\newcommand{\frkm}{{\mathfrak m}}    \newcommand{\frkM}{{\mathfrak M}}
\newcommand{\frkn}{{\mathfrak n}}    \newcommand{\frkN}{{\mathfrak N}}
\newcommand{\frko}{{\mathfrak o}}    \newcommand{\frkO}{{\mathfrak O}}
\newcommand{\frkp}{{\mathfrak p}}    \newcommand{\frkP}{{\mathfrak P}}
\newcommand{\frkq}{{\mathfrak q}}    \newcommand{\frkQ}{{\mathfrak Q}}
\newcommand{\frkr}{{\mathfrak r}}    \newcommand{\frkR}{{\mathfrak R}}
\newcommand{\frks}{{\mathfrak s}}    \newcommand{\frkS}{{\mathfrak S}}
\newcommand{\frkt}{{\mathfrak t}}    \newcommand{\frkT}{{\mathfrak T}}
\newcommand{\frku}{{\mathfrak u}}    \newcommand{\frkU}{{\mathfrak U}}
\newcommand{\frkv}{{\mathfrak v}}    \newcommand{\frkV}{{\mathfrak V}}
\newcommand{\frkw}{{\mathfrak w}}    \newcommand{\frkW}{{\mathfrak W}}
\newcommand{\frkx}{{\mathfrak x}}    \newcommand{\frkX}{{\mathfrak X}}
\newcommand{\frky}{{\mathfrak y}}    \newcommand{\frkY}{{\mathfrak Y}}
\newcommand{\frkz}{{\mathfrak z}}    \newcommand{\frkZ}{{\mathfrak Z}}

%math boldface latters

\newcommand{\bfa}{{\mathbf a}}    \newcommand{\bfA}{{\mathbf A}}
\newcommand{\bfb}{{\mathbf b}}    \newcommand{\bfB}{{\mathbf B}}
\newcommand{\bfc}{{\mathbf c}}    \newcommand{\bfC}{{\mathbf C}}
\newcommand{\bfd}{{\mathbf d}}    \newcommand{\bfD}{{\mathbf D}}
\newcommand{\bfe}{{\mathbf e}}    \newcommand{\bfE}{{\mathbf E}}
\newcommand{\bff}{{\mathbf f}}    \newcommand{\bfF}{{\mathbf F}}
\newcommand{\bfg}{{\mathbf g}}    \newcommand{\bfG}{{\mathbf G}}
\newcommand{\bfh}{{\mathbf h}}    \newcommand{\bfH}{{\mathbf H}}
\newcommand{\bfi}{{\mathbf i}}    \newcommand{\bfI}{{\mathbf I}}
\newcommand{\bfj}{{\mathbf j}}    \newcommand{\bfJ}{{\mathbf J}}
\newcommand{\bfk}{{\mathbf k}}    \newcommand{\bfK}{{\mathbf K}}
\newcommand{\bfl}{{\mathbf l}}    \newcommand{\bfL}{{\mathbf L}}
\newcommand{\bfm}{{\mathbf m}}    \newcommand{\bfM}{{\mathbf M}}
\newcommand{\bfn}{{\mathbf n}}    \newcommand{\bfN}{{\mathbf N}}
\newcommand{\bfo}{{\mathbf o}}    \newcommand{\bfO}{{\mathbf O}}
\newcommand{\bfp}{{\mathbf p}}    \newcommand{\bfP}{{\mathbf P}}
\newcommand{\bfq}{{\mathbf q}}    \newcommand{\bfQ}{{\mathbf Q}}
\newcommand{\bfr}{{\mathbf r}}    \newcommand{\bfR}{{\mathbf R}}
\newcommand{\bfs}{{\mathbf s}}    \newcommand{\bfS}{{\mathbf S}}
\newcommand{\bft}{{\mathbf t}}    \newcommand{\bfT}{{\mathbf T}}
\newcommand{\bfu}{{\mathbf u}}    \newcommand{\bfU}{{\mathbf U}}
\newcommand{\bfv}{{\mathbf v}}    \newcommand{\bfV}{{\mathbf V}}
\newcommand{\bfw}{{\mathbf w}}    \newcommand{\bfW}{{\mathbf W}}
\newcommand{\bfx}{{\mathbf x}}    \newcommand{\bfX}{{\mathbf X}}
\newcommand{\bfy}{{\mathbf y}}    \newcommand{\bfY}{{\mathbf Y}}
\newcommand{\bfz}{{\mathbf z}}    \newcommand{\bfZ}{{\mathbf Z}}

%caligraphic letters

\newcommand{\cala}{{\mathcal A}}
\newcommand{\calb}{{\mathcal B}}
\newcommand{\calc}{{\mathcal C}}
\newcommand{\cald}{{\mathcal D}}
\newcommand{\cale}{{\mathcal E}}
\newcommand{\calf}{{\mathcal F}}
\newcommand{\calg}{{\mathcal G}}
\newcommand{\calh}{{\mathcal H}}
\newcommand{\cali}{{\mathcal I}}
\newcommand{\calj}{{\mathcal J}}
\newcommand{\calk}{{\mathcal K}}
\newcommand{\call}{{\mathcal L}}
\newcommand{\calm}{{\mathcal M}}
\newcommand{\caln}{{\mathcal N}}
\newcommand{\calo}{{\mathcal O}}
\newcommand{\calp}{{\mathcal P}}
\newcommand{\calq}{{\mathcal Q}}
\newcommand{\calr}{{\mathcal R}}
\newcommand{\cals}{{\mathcal S}}
\newcommand{\calt}{{\mathcal T}}
\newcommand{\calu}{{\mathcal U}}
\newcommand{\calv}{{\mathcal V}}
\newcommand{\calw}{{\mathcal W}}
\newcommand{\calx}{{\mathcal X}}
\newcommand{\caly}{{\mathcal Y}}
\newcommand{\calz}{{\mathcal Z}}

%math script

\newcommand{\scra}{{\mathscr A}}
\newcommand{\scrb}{{\mathscr B}}
\newcommand{\scrc}{{\mathscr C}}
\newcommand{\scrd}{{\mathscr D}}
\newcommand{\scre}{{\mathscr E}}
\newcommand{\scrf}{{\mathscr F}}
\newcommand{\scrg}{{\mathscr G}}
\newcommand{\scrh}{{\mathscr H}}
\newcommand{\scri}{{\mathscr I}}
\newcommand{\scrj}{{\mathscr J}}
\newcommand{\scrk}{{\mathscr K}}
\newcommand{\scrl}{{\mathscr L}}
\newcommand{\scrm}{{\mathscr M}}
\newcommand{\scrn}{{\mathscr N}}
\newcommand{\scro}{{\mathscr O}}
\newcommand{\scrp}{{\mathscr P}}
\newcommand{\scrq}{{\mathscr Q}}
\newcommand{\scrr}{{\mathscr R}}
\newcommand{\scrs}{{\mathscr S}}
\newcommand{\scrt}{{\mathscr T}}
\newcommand{\scru}{{\mathscr U}}
\newcommand{\scrv}{{\mathscr V}}
\newcommand{\scrw}{{\mathscr W}}
\newcommand{\scrx}{{\mathscr X}}
\newcommand{\scry}{{\mathscr Y}}
\newcommand{\scrz}{{\mathscr Z}}

%math Bbb

\newcommand{\AAA}{{\mathbb A}} %not \AA
\newcommand{\BB}{{\mathbb B}}
\newcommand{\CC}{{\mathbb C}}
\newcommand{\DD}{{\mathbb D}}
\newcommand{\EE}{{\mathbb E}}
\newcommand{\FF}{{\mathbb F}}
\newcommand{\GG}{{\mathbb G}}
\newcommand{\HH}{{\mathbb H}}
\newcommand{\II}{{\mathbb I}}
\newcommand{\JJ}{{\mathbb J}}
\newcommand{\KK}{{\mathbb K}}
\newcommand{\LL}{{\mathbb L}}
\newcommand{\MM}{{\mathbb M}}
\newcommand{\NN}{{\mathbb N}}
\newcommand{\OO}{{\mathbb O}}
\newcommand{\PP}{{\mathbb P}}
\newcommand{\QQ}{{\mathbb Q}}
\newcommand{\RR}{{\mathbb R}}
\newcommand{\SSS}{{\mathbb S}} %not \SS
\newcommand{\TT}{{\mathbb T}}
\newcommand{\UU}{{\mathbb U}}
\newcommand{\VV}{{\mathbb V}}
\newcommand{\WW}{{\mathbb W}}
\newcommand{\XX}{{\mathbb X}}
\newcommand{\YY}{{\mathbb Y}}
\newcommand{\ZZ}{{\mathbb Z}}

\newcommand{\phm}{\phantom}
\newcommand{\ds}{\displaystyle }
\newcommand{\smallstrut}{\vphantom{\vrule height 3pt }}
\def\bdm #1#2#3#4{\left(
\begin{array} {c|c}{\ds{#1}}
 & {\ds{#2}} \\ \hline
{\ds{#3}\vphantom{\ds{#3}^1}} &  {\ds{#4}}
\end{array}
\right)}
\newcommand{\wtd}{\widetilde }
\newcommand{\bsl}{\backslash }
\newcommand{\GL}{{\mathrm{GL}}}
\newcommand{\SL}{{\mathrm{SL}}}
\newcommand{\GSp}{{\mathrm{GSp}}}
\newcommand{\PGSp}{{\mathrm{PGSp}}}
\newcommand{\SP}{{\mathrm{Sp}}}
\newcommand{\SO}{{\mathrm{SO}}}
\newcommand{\SU}{{\mathrm{SU}}}
\newcommand{\Ind}{\mathrm{Ind}}
\newcommand{\Hom}{{\mathrm{Hom}}}
\newcommand{\Ad}{{\mathrm{Ad}}}
\newcommand{\Sym}{{\mathrm{Sym}}}
\newcommand{\Mat}{\mathrm{M}}
\newcommand{\sgn}{\mathrm{sgn}}
\newcommand{\trs}{\,^t\!}
\newcommand{\iu}{\sqrt{-1}}
\newcommand{\oo}{\hbox{\bf 0}}
\newcommand{\ono}{\hbox{\bf 1}}
\newcommand{\smallcirc}{\lower .3em \hbox{\rm\char'27}\!}
\newcommand{\bAf}{\bA_{\hbox{\eightrm f}}}
\newcommand{\thalf}{{\textstyle{\frac12}}}
\newcommand{\shp}{\hbox{\rm\char'43}}
\newcommand{\Gal}{\operatorname{Gal}}

\newcommand{\bdel}{{\boldsymbol{\delta}}}
\newcommand{\bchi}{{\boldsymbol{\chi}}}
\newcommand{\bgam}{{\boldsymbol{\gamma}}}
\newcommand{\bome}{{\boldsymbol{\omega}}}
\newcommand{\bpsi}{{\boldsymbol{\psi}}}
\newcommand{\GK}{\mathrm{GK}}
\newcommand{\ord}{\mathrm{ord}}
\newcommand{\diag}{\mathrm{diag}}
\newcommand{\ua}{{\underline{a}}}
\newcommand{\ZZn}{\ZZ_{\geq 0}^n}
\newcommand{\calhnd}{{\mathcal H}^\mathrm{nd}}
\newcommand{\EGK}{\mathrm{EGK}}

\theoremstyle{plain}
\newtheorem{theorem}{Theorem}[section]
\newtheorem{lemma}{Lemma}[section]
\newtheorem{proposition}{Proposition}[section]
\newtheorem{corollary}{\textbf{Corollary}}[section]
\theoremstyle{definition}
\newtheorem{definition}{Definition}[section]
\newtheorem{conjecture}{Conjecture}[section]
\newtheorem{remark}{\textbf{Remark}}[section]

\newtheorem{mainthm}{\textbf{Theorem}}
\renewcommand{\themainthm}{}

%%%%%%%%%%%%%%%%title%%%%%%%%%%%%%%%%%%%
\title[]{On the Gross-Keating invariant of a quadratic form over a non-archimedean local field}
\author{Tamotsu IKEDA}
\address{Graduate school of mathematics, Kyoto University, Kitashirakawa, Kyoto, 606-8502, Japan}
\email{ikeda@math.kyoto-u.ac.jp}
\author{Hidenori KATSURADA}
\address{Muroran Institute of Technology
27-1 Mizumoto, Muroran, 050-8585, Japan}
\email{hidenori@mmm.muroran-it.ac.jp}
\begin{abstract}
Let $B$ be a half-integral symmetric matrix of size $n$ defined over $\QQ_p$.
The Gross-Keating invariant of $B$ was defined by Gross and Keating, and has  important applications to arithmetic geometry.
But the nature of the Gross-Keating invariant was not understood very well for $n\geq 4$.
In this paper, we establish basic properties of the Gross-Keating invariant of a half-integral symmetric matrix of general size over an arbitrary non-archimedean local field of characteristic zero.
\end{abstract}
\maketitle

%%%%%%%%%%%%%%%%%%%%%%%%%%%%%%

\section*{Introduction}

Gross and Keating \cite{GK} introduced a certain invariant for a quadratic form over $\ZZ_p$.
This invariant is called the Gross-Keating invariant, and has applications to arithmetic geometry (see ARGOS seminar \cite{ARGOS}, Bouw \cite{Bouw}, Gross and Keating \cite{GK}, Kudla, Rapoport, and Yang \cite{kry}, Wedhorn \cite{Wedhorn}).
For $p\neq 2$, the Gross-Keating invariant can be easily calculated by means of the Jordan splitting.
For $p=2$, the nature of the Gross-Keating invariant of a quadratic form of degree $n$ was understood only for $n\leq 3$.
The purpose of this paper is to investigate the basic properties of the Gross-Keating invariants for a quadratic form of general degree over the ring of integers of a non-archimedean local field of characteristic zero.

Let us recall the definition of the Gross-Keating invariant.
Let $F$ be a non-archimedean local field of characteristic $0$, and $\frko=\frko_F$ its ring of integers.
$F$ is said to be dyadic if $F$ is a finite extension of $\QQ_2$.
The order $\ord(x)$ of $x\in F^\times$ is normalized so that $\ord(\vpi)=1$ for a prime element $\vpi$ of $F$.
We understand $\ord(0)=+\infty$.

The set of symmetric matrices $B\in \mathrm{M}_n(F)$ of size $n$ is denoted by $\mathrm{Sym}_n(F)$.
For $B\in \mathrm{Sym}_n(F)$ and $X\in\GL_n(F)$, we set $B[X]={}^t\! XBX$.
We say that $B=(b_{ij})\in \mathrm{Sym}_n(F)$ is a half-integral symmetric matrix if
\begin{align*}
b_{ii}\in\frko_F &\qquad (1\leq i\leq n),  \\
2b_{ij}\in\frko_F& \qquad (1\leq i\leq j\leq n).
\end{align*}
The set of all half-integral symmetric matrices of size $n$ is denoted by $\calh_n(\frko)$.
An element $B\in\calh_n(\frko)$ is non-degenerate if $\det B\neq 0$.
The set of all non-degenerate elements of $\calh_n(\frko)$ is denoted by $\calhnd_n(\frko)$.

The equivalence class of $B\in\calh_n(\frko)$ is denoted by $\{B\}$, i.e., 
$\{B\}=\{B[U]\,|\, U\in\GL_n(\frko)\}$.
We write $B\sim B'$ if $B'\in\{B\}$.

\begin{definition} % Definition 0.1
Let $B=(b_{ij})\in\calhnd_n(\frko)$.
Let $S(B)$ be the set of all non-decreasing sequences $(a_1, \ldots, a_n)\in\ZZn$ such that
\begin{align*}
&\ord(b_{ii})\geq a_i \qquad\qquad\qquad\quad (1\leq i\leq n), \\
&\ord(2 b_{ij})\geq (a_i+a_j)/2  \qquad\; (1\leq i\leq j\leq n).
\end{align*}
Put
\[
\bfS(\{B\})=\bigcup_{B'\in\{B\}} S(B')=\bigcup_{U\in\GL_n(\frko)} S(B[U]).
\]
The Gross-Keating invariant $\GK(B)$ of $B$ is the greatest element of $\bfS(\{B\})$ with respect to the lexicographic order $\succeq$ on $\ZZn$.
\end{definition}

Note that for $\ua\in S(B)$ and $\sig\in\frkS_n$, we have $\mathrm{ord}(2b_{i\sig(j)}))\geq (a_i+a_{\sig(j)})/2$.
It follows that
\[
\ord(2^n b_{1\sig(1)}\cdots b_{n\sig(n)})\geq a_1+\cdots+a_n.
\]
Therefore we have $a_1+\cdots +a_n\leq \ord(\det(2B))$.
In particular, $\bfS(\{B\})$ is a finite set.

A sequence of length $0$ is denoted by $\emptyset$.
When $B$ is the empty matrix, we understand $\GK(B)=\emptyset$.
By definition, the Gross-Keating invariant $\GK(B)$ is determined only by the equivalence class of $B$.
Note that $\GK(B)=(a_1, \ldots, a_n)$ is also defined by
\begin{align*}
a_1&=\max_{(y_1,  \ldots)\in \bfS(\{B\})} \,\{y_1\}, \\
a_2&=\max_{(a_1, y_2, \ldots)\in \bfS(\{B\})}\, \{y_2\}, \\
&\cdots \\
a_n&=\max_{(a_1, a_2, \ldots, a_{n-1}, y_n)\in \bfS(\{B\})}\, \{y_n\}.
\end{align*}

\begin{definition} % Definition 0.2
\label{def:0.2}
$B\in\calh^\mathrm{nd}_n(\frko)$ is optimal if $\GK(B)\in S(B)$.
\end{definition}
By definition, a non-degenerate half-integral symmetric matrix $B\in\calhnd_n(\frko)$ is equivalent to an optimal form.

\begin{remark} % Remark 0.1
If $F$ is non-dyadic, one can easily show that the Jordan splitting of $B$ is optimal (See Remark \ref{rem:1.1}).
On the other hand, if $F$ is dyadic, a Jordan splitting may not be optimal.
For example, if $F$ is dyadic, then $B=\begin{pmatrix} 1 & 0 \\ 0 & -1\end{pmatrix}$ is not optimal (See section \ref{sec:2}).
Thus the issue is the case when $F$ is dyadic. 
A characterization of an optimal form will be given in section \ref{sec:5} (See Theorem \ref{thm:5.1a}).
\end{remark}

For $B\in\calhnd_n(\frko)$, we put $D_B=(-4)^{[n/2]}\det B$.
$D_B$ (or its image in $F^\times/F^{\times 2}$) is often called the signed determinant (\cite{lam}) or the discriminant (\cite{scharlau}) of $2B$.
Here, $F^{\times 2}=\{x^2\,|\, x\in F^\times\}$.
If $n$ is even, we denote the discriminant ideal of $F(\sqrt{D_B})/F$ by $\mathfrak{D} _B$.
We also put
\[
\xi_B=
\begin{cases} 
1 & \text{ if $D_B\in F^{\times 2}$,} \\
-1 & \text{ if $F(\sqrt{D_B})/F$ is unramified and $[F(\sqrt{D_B}):F]=2$,} \\
0 & \text{ if $F(\sqrt{D_B})/F$ is ramified.} 
\end{cases}
\]
We also write $\xi(B)$ for $\xi_B$, if there is no fear of confusion.
\begin{definition} % Definition 0.3
\label{def:0.3}
For $B\in\calhnd_n(\frko)$, we put
\[
\Del(B)=
\begin{cases}
\ord(D_B) & \text{ if $n$ is odd,} \\
\ord(D_B)-\ord(\mathfrak{D}_B)+1-\xi_B^2 & \text{ if $n$ is even.}
\end{cases}
\]
\end{definition}
Note that if $n$ is even, then
\[
\Del(B)=
\begin{cases}
\ord(D_B) & \text{ if $\ord(\mathfrak{D}_B)=0$,} \\
\ord(D_B)-\ord(\mathfrak{D}_B)+1 & \text{ if $\ord(\mathfrak{D}_B)>0$.}
\end{cases}
\]

For $\ua=(a_1, a_2, \ldots, a_n)\in\ZZn$, we write $|\ua|=a_1+a_2+\cdots+a_n$.

\begin{theorem} % Theorem 0.1
\label{thm:0.1}
For $B\in\calhnd_n(\frko)$,  we have
\[
|\GK(B)|=\Del(B).
\]
\end{theorem}

For a non-decreasing sequence $\ua=(a_1, a_2, \ldots, a_n)\in\ZZn$, we set
\[
G_{\ua}=
\{g=(g_{ij})\in \GL_n(\frko) \,|\, \text{ $\ord(g_{ij})\geq (a_j-a_i)/2$, if $a_i<a_j$}\}.
\]

\begin{theorem} % Theorem 0.2
\label{thm:0.2}
Suppose that $B\in\calhnd_n(\frko)$ is optimal and  $\GK(B)=\ua$.
Let $U\in \GL_n(\frko)$.
Then $B[U]$ is optimal if and only if $U\in G_\ua$.
\end{theorem}

For $B=(b_{ij})_{1\leq i,j\leq n}\in\calh_n(\frko)$ and $1\leq m\leq n$,  we denote the upper left $m\times m$ submatrix $(b_{ij})_{1\leq i, j\leq m}\in\calh_m(\frko)$ by $B^{(m)}$.
For $\underline{a}=(a_1, a_2, \ldots, a_n)\in\ZZ_{\geq 0}^n$, we put $\underline{a}^{(m)}=(a_1, a_2, \ldots, a_m)$ for $m\leq n$.

\begin{theorem} % Theorem 0.3
\label{thm:0.3}
Suppose that $B\in\calhnd_n(\frko)$ is optimal and  $\GK(B)=\ua$.
If $a_k<a_{k+1}$, then $B^{(k)}$ is also optimal and $\GK(B^{(k)})=\ua^{(k)}$.
\end{theorem}

\begin{definition} % Definition 0.4
\label{def:0.4}
The Clifford invariant (see Scharlau \cite{scharlau}, p.~333) of $B\in\calhnd_n(\frko)$ is the Hasse invariant of the Clifford algebra (resp.~the even Clifford algebra) of $B$ if $n$ is even (resp.~odd).
\end{definition}
We denote the Clifford invariant of $B$ by $\eta_B$.
We also write $\eta(B)$ for $\eta_B$, if there is no fear of confusion.
If $B$ is $\GL_n(F)$-equivalent to a diagonal matrix $\mathrm{diag}(b'_1, \ldots, b'_n)$, then 
\begin{align*}
\eta_B
=&
\langle -1, -1 \rangle^{[(n+1)/4]}\langle -1, \det B \rangle^{[(n-1)/2]} \prod_{i < j} \langle b'_i, b'_j \rangle \\
=&\begin{cases}
\langle -1, -1 \rangle^{m(m-1)/2}\langle -1, \det B \rangle^{m-1} \ds\prod_{i < j} \langle b'_i, b'_j \rangle & \text{ if $n=2m$, } \\
\noalign{\vskip 6pt}
\langle -1, -1 \rangle^{m(m+1)/2}\langle -1, \det B \rangle^{m} \ds\prod_{i < j} \langle b'_i, b'_j \rangle & \text{ if $n=2m+1$. } 
\end{cases}
\end{align*}
(See Scharlau \cite{scharlau} pp.~80--81.)
%If $H\in\calhnd_2(\frko)$ is $\GL_2(F)$-isomorphic to a hyperbolic plane, then $\eta_{B\perp H}=\eta_B$.
The Clifford invariant $\eta_B$ depends only on the image of $B$ in the Witt group of $F$.
In particular, if $n$ is odd, then we have
\begin{align*}
\eta_B
=&
\begin{cases} 
1 & \text{ if  $B$ is split over $F$,} \\
-1 & \text{ otherwise.} 
\end{cases}
\end{align*}

\begin{theorem} % Theorem 0.4
\label{thm:0.4}
Let $B, B_1\in\calhnd_n(\frko)$.
Suppose that $B\sim B_1$ and both $B$ and $B_1$ are optimal.
Let $\ua=(a_1, a_2, \ldots, a_n)=\GK(B)=\GK(B_1)$.
Suppose that $a_k<a_{k+1}$ for $1\leq k < n$.
Then the following assertions (1) and (2) hold.
\begin{itemize}
\item[(1)] If $k$ is even, then $\xi_{B^{(k)}}=\xi_{B_1^{(k)}}$.
\item[(2)] If $k$ is odd, then $\eta_{B^{(k)}}=\eta_{B_1^{(k)}}$.
\end{itemize}
\end{theorem}

\begin{remark} % Remark 0.2
It is known that $B^{(k)}\sim B_1^{(k)}$ if $F$ is non-dyadic.
But it is not true if $F$ is dyadic.
Suppose that $F=\QQ_2$.
Put $B=\begin{pmatrix} 1 & 0 \\ 0 & 2 \end{pmatrix}$ and $B_1=\begin{pmatrix} 3 & 0 \\ 0 & 6 \end{pmatrix}$.
Then we have $B\sim B_1$ and $\GK(B)=(0,1)$.
Moreover both $B$ and $B_1$ are optimal.
But $B^{(1)}=(1)$ is not equivalent to $B_1^{(1)}=(3)$.
Thus $B^{(k)}$ and $B_1^{(k)}$ may not be equivalent if $F$ is dyadic.
\end{remark}

We define $a^\ast_1<a^\ast_2< \cdots< a^\ast_r$ by
\[
\{a^\ast_1\ldots, a^\ast_r\}=\{a_1, a_2, \ldots, a_n\}.
\]
Put $n_s=\sharp\{i\,|\, a_i=a^\ast_s\}$ and  $n^\ast_s=n_1+n_2+\cdots+n_s$ for $s=1,\ldots, r$.
In particular, $n^\ast_r=n$.
Put
\[
\zeta_s=
\begin{cases}
\xi_{B^{(n_s^\ast)}} & \text{ if $n_s^\ast$ is even,} \\
\eta_{B^{(n_s^\ast)}} & \text{ if $n_s^\ast$ is odd.} 
\end{cases}
\]
By definition, if $n^\ast_s$ is even, we have
\[
{\zeta_s \not=0}
\quad
\Longleftrightarrow
\quad
{a^\ast_1 n_1+\cdots+a^\ast_s n_s}\text{ is even.}
\]
Moreover,  $\zeta_s \neq 0$ if $n^\ast_s$ is odd.
Then we can show the following theorem (See Theorem \ref{thm:6.1}).
\begin{theorem}
\label{thm:0.5}
Suppose that $n^\ast_s$ is odd. 
Then we have 
\begin{itemize}
\item[(a)] Assume that $n^\ast_i$ is even for any $i<s$.
Then we have
\[
\zeta_s
=\zeta_1^{a^\ast_1+a^\ast_2}
\zeta_2^{a^\ast_2+a^\ast_3}
\cdots
\zeta_{s-1}^{a^\ast_{s-1}+a^\ast_s}.
\]
In particular, $\zeta_1=1$ if $n_1$ is odd.
\item[(b)] 
Assume that $a^\ast_1n_1+\cdots + a^\ast_{s-1}n_{s-1}+a^\ast_s(n_s-1)$ is even and that $n^\ast_i$ is odd for some $i<s$.
Let $t<s$ be the largest number such that $n^\ast_t$ is odd.
Then we have
\[
\zeta_s=
\zeta_t
\zeta_{t+1}^{a^\ast_{t+1}+a^\ast_{t+2}} 
\zeta_{t+2}^{a^\ast_{t+2} + a^\ast_{t+3}}
\cdots
\zeta_{s-1}^{a^\ast_{s-1}+a^\ast_s}.
\]
In particular, $\zeta_s=\zeta_t$ if $t+1=s$.
\end{itemize}
\end{theorem}
The datum $\mathrm{EGK}(B)=(n_1, \ldots, n_r; a^\ast_1, \ldots, a^\ast_r; \zeta_1, \ldots, \zeta_r)$ is called the extended GK datum of $B$.
In general, a datum $(n_1, \ldots, n_r; m_1, \ldots, m_r; \zeta_1, \ldots, \zeta_r)$ 
 satisfying these conditions is called an EGK datum (See Definition \ref{def:6.2}).
The EGK data have an application to the theory of Siegel series.
For the theory of Siegel series, one can consult Katsurada \cite{Katsurada} and Shimura \cite{shimura97}.
The Siegel series arises from the local factor of the Fourier coefficients of Siegel Eisenstein series.
For $F=\QQ_p$, an explicit formula for the Siegel series was given by Katsurada \cite{Katsurada}.
But the explicit formula in \cite{Katsurada} was very complicated for $p=2$.
In out forthcoming paper \cite{ikedakatsurada}, we will show that there exists a Laurent polynomial $\tilde\calf(\mathrm{EGK}(B); Y, X)\in\ZZ[X^{1/2}, X^{-1/2}, Y^{1/2}, Y^{-1/2}]$ such that the Laurent polynomial $\tilde F(B, X)$ attached to the Siegel series of $B$ is given by
\[
\widetilde F(B,X)=\widetilde \calf(\mathrm{EGK}(B);q^{1/2},X).
\]
In particular, the Siegel series of $B$ is determined by $\mathrm{EGK}(B)$.
Note that this formula holds for both the non-dyadic case and the dyadic case.

We now explain the content of this paper.
In section \ref{sec:1}, we will discuss some elementary properties of the Gross-Keating invariant.
In particular, we show that if $B_1\in\calhnd_m(\frko)$ is represented by $B\in\calhnd_n(\frko)$, then $\GK(B_1)\succeq \GK(B)^{(m)}$ (Lemma \ref{lem:1.2}).
This lemma is useful to calculate Gross-Keating invariants.
In section \ref{sec:2}, we calculate Gross-Keating invariants of binary forms explicitly.
The results of section \ref{sec:3} and \ref{sec:4} are the technical heart of this paper.
In section \ref{sec:3}, we introduce reduced forms (see Definition \ref{def:3.2}) and discuss its properties.
In section \ref{sec:4}, we prove the reduction theorem (Theorem \ref{thm:4.1}) which says that any half-integral symmetric matrix is equivalent to a reduced form.
By the reduction theorem, the proofs of the theorems above are reduced to the case of reduced form.
Using these results, we prove the Theorems \ref{thm:0.1}--\ref{thm:0.4} in section \ref{sec:5}.
In section \ref{sec:6}, we discuss some combinatorial properties of auxiliary invariants $\xi_{B^{(k)}}$ and $\eta_{B^{(k)}}$.
To this end, we introduce EGK data (Definition 6.2), and show that these invariants satisfy the axioms of EGK data (Theorem 6.1).

Throughout this paper, except for section \ref{sec:1} and section \ref{sec:6},  we mainly discuss the dyadic case.
The proof of the Theorems \ref{thm:0.1}--\ref{thm:0.4} for the non-dyadic case is briefly explained at the end of section \ref{sec:1}.

We thank the referee for many useful comments.
We also thank Sungmun Cho for his comments.
This research was partially supported by the JSPS KAKENHI Grant Number 26610005, 24540005.

\section*{Notation}

When $R$ is a ring, the set of $m\times n$ matrices with entries in $R$ is denoted by $\mathrm{M}_{mn}(R)$ or $\mathrm{M}_{m,n}(R)$.
As usual, $\mathrm{M}_n(R)=\mathrm{M}_{n,n}(R)$.
The identity matrix of size $n$ is denoted by $\mathbf{1}_n$.
For $X_1\in \mathrm{M}_s(R)$ and $X_2\in\mathrm{M}_t(R)$, the matrix $\begin{pmatrix} X_1 & 0 \\ 0 & X_2\end{pmatrix}\in\mathrm{M}_{s+t}(R)$ is denoted by $X_1\perp X_2$.
The diagonal matrix whose diagonal entries are $b_1$, $\ldots$, $b_n$ is denoted by $\mathrm{diag}(b_1, \dots, b_n)=(b_1)\perp\dots\perp (b_n)$.

Let $F$ be a non-archimedean local field of characteristic $0$, and $\frko=\frko_F$ its ring of integers.
The maximal ideal and  the residue field of $\frko$ are denoted by $\frkp$ and $\frkk$, respectively.
We put $q=[\frko:\frkp]$.
We fix a prime element $\vpi$ of $\frko$ once and for all.
The order of $x\in F^\times$ is given by $\mathrm{ord}(x)=n$ for $x\in \vpi^n \frko^\times$.
We understand $\ord(0)=+\infty$.
Put $F^{\times 2}=\{x^2\,|\, x\in F^\times\}$ and $\frko^{\times 2}=\{x^2\,|\, x\in\frko^\times\}$.

When $G$ is a subgroup of $\GL_n(F)$, we shall say that two elements $B_1, B_2\in\mathrm{Sym}_n(F)$ are called $G$-equivalent, if there is an element $X\in G$ such that $B_1[X]=B_2$.
When $G=\GL_n(\frko)$, we just say they are equivalent.

The lexicographic order $\succeq$ on $\ZZ_{\geq 0}^n$ is, as usual, defined as follows.
For distinct sequences $(y_1, y_2, \ldots, y_n),  (z_1, z_2, \ldots, z_n)\in \ZZ_{\geq 0}^n$, let $j$ be the largest integer such that $y_i=z_i$ for $i<j$.
Then $(y_1, y_2, \ldots, y_n)\succneqq  (z_1, z_2, \ldots, z_n)$ if $y_j>z_j$.
We define $(y_1, y_2, \ldots, y_n)\succeq  (z_1, z_2, \ldots, z_n)$ if $(y_1, y_2, \ldots, y_n)\succneqq  (z_1, z_2, \ldots, z_n)$ or $(y_1, y_2, \ldots, y_n) = (z_1, z_2, \ldots, z_n)$.

\section{Elementary properties of the Gross-Keating invariant}
\label{sec:1}

Let $L$ be a free module of rank $n$ over $\frko$, and $Q$ an $\frko$-valued quadratic form on $L$.
The pair $(L, Q)$ is called a quadratic module over $\frko$.
The symmetric bilinear form $(x, y)_Q$ associated to $Q$ is defined by
\[
 (x, y)_Q=Q(x+y)-Q(x)-Q(y), \qquad x, y\in L.
\]
When there is no fear of confusion, $(x, y)_Q$ is simply denoted by $(x, y)$.
If $\underline{\psi}=(\psi_1, \ldots, \psi_n)$ is an ordered basis of $L$, we call the triple $(L, Q, \underline{\psi})$ a framed quadratic $\frko$-module.
Hereafter, ``a basis'' means an ordered basis.
For a framed quadratic $\frko$-module $(L, Q, \underline{\psi})$, we define a matrix $B=(b_{ij})\in\calh_n(\frko)$ by
\[
 b_{ij}=\frac12 (\psi_i, \psi_j).
\]
The isomorphism class of $(L, Q, \underline{\psi})$ (as a framed quadratic $\frko$-module) is determined by $B$.
We say that $B\in\calh_n(\frko)$ is associated to the framed quadratic module $(L, Q, \underline{\psi})$.
If $B$ is non-degenerate, we also say $(L, Q)$ or $(L, Q, \underline{\psi})$ is non-degenerate.
The set $S(B)$ is also denoted by $S(\underline{\psi})$.
If $B$ is optimal, then $\underline{\psi}$ is called an optimal basis.
We consider $\mathrm{Aut}(L)$ acting on $L$ from the right.
When $U\in \mathrm{Aut}(L)$ is given by $\psi_j\mapsto \sum_{i=1}^n \psi_i u_{ij}$, with $(u_{ij})\in\GL_n(\frko)$, 
we define an ordered basis $\underline{\psi}U=((\underline{\psi}U)_1, \ldots, (\underline{\psi}U)_n)$ by $(\underline{\psi}U)_j=\sum_{i=1}^n \psi_i u_{ij}$.
Then the matrix associated to $(L, Q, \underline{\psi}U)$ is equal to $B[U]=B[(u_{ij})]$.
In particular, the equivalence class of $B$ is determined by the isomorphism class of the quadratic module $(L, Q)$.
The norm $\mathrm{n}(L)$ of $(L, Q)$ is the fractional ideal generated by $\{Q(x)\,|\, x\in L\}$.
It is known (see \cite{thyang} Lemma B.1) that $a_1=\mathrm{ord}(\mathrm{n}(L))$, where $a_1$ is the first entry of $\GK(B)$.

Let $\ua=(a_1, \ldots, a_n)\in\ZZ_{\geq 0}^n$ be a non-decreasing sequence.
We define $n_1, n_2, \dots, n_r$ with $n_1+\cdots+n_r=n$ by
\begin{align*}
&a_1=\cdots=a_{n_1} <a_{n_1+1}, \\
&a_{n_1}<a_{n_1+1}=\cdots=a_{n_1+n_2}<a_{n_1+n_2+1}, \\
&\cdots \\
&a_{n_1+\cdots+n_{r-1}}<a_{n_1+\cdots+n_{r-1}+1}=\cdots =a_{n_1+\cdots+n_r}.
\end{align*}
For $s=1, 2, \ldots, r$, we set
\[ 
n^\ast_s=\sum_{u=1}^{s} n_u.
\]
We put $n^\ast_0=0$.
The $s$-th block $I_s$ is defined by $I_s=\{n^\ast_{s-1}+1, n^\ast_{s-1}+2, \ldots, n^\ast_s\}$ for $s=1, 2, \dots, r$.
We put $a^\ast_s=a_{n^\ast_{s-1}+1}=\cdots=a_{n^\ast_s}$.

Let $(L, Q, \underline{\psi})$ be the framed quadratic $\frko$-module associated to  $B=(b_{ij})$.
For $s=1, \ldots, r$, we denote by $L_s$ the submodule of $L$ generated by $\{\psi_k\,|\, n^\ast_{s-1}+1\leq k\leq n\}=\{\psi_k\,|\, k\in I_s\cup\cdots \cup I_r \}$.
We put $L_{r+1}=\{0\}$.

Let $S^0(B)$ be the set of all non-decreasing sequences $(a_1, \ldots, a_n)\in\ZZn$ such that
\begin{align*}
&\ord(b_{ii})> a_i \qquad\qquad\qquad\quad (1\leq i\leq n), \\
&\ord(2 b_{ij})> (a_i+a_j)/2  \qquad\; (1\leq i\leq j\leq n).
\end{align*}

\begin{lemma} % Lemma 1.1
\label{lem:1.1}
Suppose that $\ua=(a_1, \ldots, a_n)\in S(B)$ with $B\in\calhnd_n(\frko)$.
Let $(L,Q, \underline{\psi})$ and $L_1, \ldots, L_r$ be as above.
If $x\in L_s$ and $y\in L_t$, then we have
\[
\ord(Q(x))\geq a^\ast_s, \qquad
\ord((x, y))\geq \frac{a^\ast_s+a^\ast_t}2.
\]
Moreover, if $\ua\in S^0(B)$, then we have
\[
\ord(Q(x))> a^\ast_s, \qquad
\ord((x, y))> \frac{a^\ast_s+a^\ast_t}2
\]
for $x\in L_s$ and $y\in L_t$.
\end{lemma}
\begin{proof}
The proof of this lemma is easy and is left to the reader.
\end{proof}

Recall that $B_1\in\calh_m(\frko)$ is represented by $B\in\calh_n(\frko)$, if there exists $X\in \mathrm{M}_{nm}(\frko)$ such that $B_1={}^t \! XBX$.
For $\underline{a}=(a_1, a_2, \ldots, a_n)\in\ZZ_{\geq 0}^n$, we put $\underline{a}^{(m)}=(a_1, a_2, \ldots, a_m)$ for $m\leq n$.
\begin{lemma} % Lemma 1.2
\label{lem:1.2}
Suppose that $\ua=(a_1, \ldots, a_n)\in\ZZ_{\geq 0}^n$ is a non-decreasing sequence and $\ua\in S(B)$ with $B\in\calhnd_n(\frko)$.
If $B_1\in\calhnd_m(\frko)$ is represented by $B$, then we have $\ua^{(m)}\in \bfS(\{B_1\})$.
In particular, $\GK(B_1)\succeq\ua^{(m)}$
\end{lemma}
\begin{proof}
Let $(L, Q, \underline{\psi})$ be the framed quadratic module corresponding to $B$.
We define $L_1, \dots, L_{r+1}$ as in Lemma \ref{lem:1.1}.
Let $(L_{B_1}, Q_1)$ be the quadratic module corresponding to $B_1$.
We may consider $L_{B_1}$ as a submodule of $L$.
It is enough to find an ordered basis $\underline{\phi}=(\phi_1, \ldots, \phi_m)$ of $L_{B_1}$ such that  $\ua^{(m)}\in S(\underline{\phi})$.
Put $M=L_{B_1}$ and $M_u=M\cap L_u$ for $u=1, \ldots, r+1$.
Then we have
\[
M=M_1\supset M_2\supset \cdots \supset M_{r+1}=\{0\}.
\]
Note that $M_u/M_{u+1}\subset L_u/L_{u+1}$ is a  torsion free $\frko$-module for $u=1,\dots, r$.
Put $m_u=\mathrm{rank}(M_u/M_{u+1})$ and $m^\ast_u=m_1+m_2+\cdots+m_u$.
Choose an ordered basis $\underline{\phi}=(\phi_1, \ldots, \phi_m)$ such that $\{\phi_{m^\ast_u+1}, \ldots, \phi_m\}$ is a basis of $M_u$ for $u=1,2, \dots, r$.
By Lemma \ref{lem:1.1}, we have $\ua^{(m)}\in S(\underline{\phi})$, since $(a_1, a_2, \ldots, a_n)$ is a non-decreasing sequence.
\end{proof}

Lemma \ref{lem:1.1} can be generalized as follows.
For $x\in\RR$, the smallest integer $n$ such that $n\geq x$ is denote by $\lceil x \rceil$.
\begin{lemma} % Lemma 1.3
\label{lem:1.3}
Suppose that $\ua=(a_1, \ldots, a_n)\in S(B)$ with $B\in\calhnd_n(\frko)$.
Let $(L,Q, \underline{\psi})$ and $L_s$ $(1\leq s\leq r)$ be as in Lemma \ref{lem:1.1}.
Put
\[
\call_s=L_s+\sum_{u=1}^{s-1} \vpi^{\lceil (a^\ast_s-a^\ast_u)/2\rceil} L_u=\sum_{u=1}^s \vpi^{\lceil (a^\ast_s-a^\ast_u)/2\rceil} L_u.
\]
If $x\in\call_s$ and $y\in \call_t$, then we have
\[
\ord(Q(x))\geq a^\ast_s, \qquad
\ord((x, y))\geq \frac{a^\ast_s+a^\ast_t}2.
\]
Moreover, if $\ua\in S^0(B)$, then we have
\[
\ord(Q(x))> a^\ast_s, \qquad
\ord((x, y))> \frac{a^\ast_s+a^\ast_t}2
\]
for $x\in \call_s$ and $y\in \call_t$.
\end{lemma}
\begin{proof}
The proof of this lemma is easy and is left to the reader.
\end{proof}

Let $\ua=(a_1, a_2, \ldots, a_n)\in\ZZn$ be a non-decreasing sequence.
Recall that we have defined the group $G_{\ua}\subset \GL_n(\frko)$ by
\[
G_{\ua}=
\{g=(g_{ij})\in \GL_n(\frko) \,|\, \text{ $\ord(g_{ij})\geq (a_j-a_i)/2$, if $a_i<a_j$}\}.
\]
We define subgroups $G_\ua^\bigtriangleup$ and $G_\ua^\bigtriangledown$ of $G_\ua$ by
\begin{align*}
G_{\ua}^\bigtriangleup&=
\{g=(g_{ij})\in G_{\ua}\,|\, \text{ $g_{ij}=0$, if $a_i> a_j$.}\}, \\
G_{\ua}^\bigtriangledown&=
\{g=(g_{ij})\in G_{\ua}\,|\, \text{ $g_{ij}=0$, if $a_i< a_j$. }\}.
\end{align*}
The symbols $\bigtriangleup$ and $\bigtriangledown$ stands for upper and lower block triangular matrices, respectively.
We also define
\begin{align*}
N_{\ua}^\bigtriangleup&=
\{g=(g_{ij})\in G_{\ua}^\bigtriangleup\,|\, \text{ $g_{ij}=\del_{ij}$, if $a_i = a_j$. }\}, \\
N_{\ua}^\bigtriangledown&=
\{g=(g_{ij})\in G_{\ua}^\bigtriangledown\,|\, \text{ $g_{ij}=\del_{ij}$, if $a_i = a_j$. }\}.
\end{align*}
Here, $\del_{ij}$ is the Kronecker delta.
Then we have $G_{\ua}=N_{\ua}^\bigtriangledown G_{\ua}^\bigtriangleup=G_{\ua}^\bigtriangleup N_{\ua}^\bigtriangledown$.

\begin{definition} % Definition 1.1
For $\ua=(a_1, \ldots, a_n)\in\ZZ_{\geq 0}^n$, put
\begin{align*}
\calm(\ua)&=\left\{(b_{ij})\in\calh_n(\frko)\,\vrule\, \begin{array}{ll}\ord(b_{ii})\geq a_i & ( 1\leq i\leq n), \\ \ord(2 b_{ij})\geq (a_i+a_j)/2 & (1\leq i < j\leq n)\end{array} \right\},  \\
\calm^0(\ua)&=\left\{(b_{ij})\in\calh_n(\frko)\,\vrule\, \begin{array}{ll}\ord(b_{ii}) > a_i & ( 1\leq i\leq n), \\ \ord(2 b_{ij}) > (a_i+a_j)/2&(1\leq i < j\leq n)\end{array}  \right\}.
\end{align*}
\end{definition}
In the definition of $\calm(\ua)$ or $\calm^0(\ua)$, we do not assume $\ua$ is non-decreasing.
Note that when $B\in\calhnd_n(\frko)$, we have
\begin{align*}
 \ua\in S(B) & \Longleftrightarrow \text{ $\ua$ is non-decreasing and $B\in\calm(\ua)$,} \\
 \ua\in S^0(B) & \Longleftrightarrow \text{ $\ua$ is non-decreasing and $B\in\calm^0(\ua)$.} 
\end{align*}

\begin{proposition} % Proposition 1.1
\label{prop:1.1}
Suppose that $\ua=(a_1, a_2, \ldots, a_n)\in\ZZn$ is a non-decreasing sequence and that $B\in\calm(\ua)$ \, $($resp. $B\in\calm^0(\ua)$ \!$)$.
Then we have $B[U]\in\calm(\ua)$ \, $($resp. $B[U]\in\calm^0(\ua)$\! $)$ for any $U\in G_\ua$.
In particular, if $B$ is optimal and $\GK(B)=\ua$, then $\GK(B[U])=\ua$ for any $U\in G_\ua$.
\end{proposition}
\begin{proof}
Suppose that $U\in G_\ua$.
Let $(L, Q, \underline{\psi})$ be the framed quadratic module associated to $B$.
Then by definition, we have $\psi_i U\in \call_s$ for  $i\in I_s$.
By Lemma \ref{lem:1.3}, we have $B[U]\in\calm(\ua)$ (resp. $B[U]\in\calm^0(\ua)$).
The proof of the last part is clear.
\end{proof}

The following lemma will be frequently used in this paper.
\begin{lemma} % Lemma 1.4
\label{lem:1.4}
Let $\ua=(a_1, a_2, \ldots, a_n)\in\ZZn$ be a non-decreasing sequence and $1\leq m<n$.
Let $s$ be the largest integer such that $a_{m+1}=\cdots=a_{m+s}$.
Put $c=a_{m+1}=\dots=a_{m+s}$.
Assume that $B\in\calm(\ua)$.
Write $B$ in a block form
\[
\begin{array}{cccc}
&\hskip -30pt
\overbrace{\hphantom{B_{11}}}^m \;
\overbrace{\hphantom{ B_{12}}}^s  \; 
\overbrace{\hphantom{ B_{15}}}^{n-m-s} \\ & 
B=\left(
\begin{array}{cccl}
B_{11} & B_{12} & B_{13} &  \\
{}^t\!B_{12} & B_{22} & B_{23}   \\
{}^t\!B_{13} &{}^t\!B_{23} & B_{33}  \\
\end{array} \hskip -10pt
\right) \hskip -5pt
\begin{array}{l}
 \left.\vphantom{B_{11}} \right\} \text{\footnotesize$m$} \\
 \left.\vphantom{B_{11}} \right\} \text{\footnotesize$s$} \\
 \left.\vphantom{B_{11}} \right\} \text{\footnotesize${n-m-s}$.} 
\end{array}
\end{array}
\]
Assume also that
\[
\begin{pmatrix}
0 & B_{12} & 0 &  \\
{}^t\!B_{12} & 0 & 0   \\
0 & 0 & 0
\end{pmatrix}
\in\calm^0(\ua),
\]
Then the following two assertions (1) and (2) hold.
\begin{itemize}
\item[(1)]
If $\GK(B_{22})\succneqq (c,c, \ldots, c)$, then $\GK(B)\succneqq \ua$.
\item[(2)]
If $\GK(B)=\ua$, then we have $\GK(B_{22})=(c,c,\ldots, c)$.
\end{itemize}
\end{lemma}
\begin{proof}
Obviously $\GK(B_{22})\succeq (c, c,  \ldots, c)$.
Suppose that $\GK(B_{22})\neq (c, c,  \ldots, c)$.
Then there exists $U\in \GL_s(\frko)$ such that
$(c, c, \dots, c, c+1)\in S(B_{22}[U])$.
Put $B'=B[\mathbf{1}_r\perp U\perp \mathbf{1}_{n-m-s}]$.
Then we have
\begin{align*}
&(a_1, a_2, \ldots, a_{m+s-1}, \underbrace{a_{m+s}+1, \ldots, a_{m+s}+1}_{n-m-s+1}) \\
=&(a_1, a_2, \ldots, a_m, \underbrace{c, c,  \ldots, c}_{s-1}, \underbrace{c+1, c+1, \ldots, c+1}_{n-m-s+1})
\in S(B').
\end{align*}
Hence the assertion (1) holds.
The assertion (2) follows from (1) immediately.
\end{proof}

\begin{remark} % Remark 1.1
\label{rem:1.1}
We briefly explain the proofs of the Theorems \ref{thm:0.1}--\ref{thm:0.4} for non-dyadic case.
Suppose that $F$ is a non-dyadic field.
Then it is well-known that any non-degenerate element $B\in\calh_n(\frko)$ is equivalent to a diagonal matrix of the form
\[
T=\diag(t_1, t_2, \ldots, t_n), \qquad \ord(t_1)\leq \ord(t_2)\leq \cdots \leq \ord(t_n),
\]
which is called a Jordan splitting of $B$.
%It is easy to see such a diagonal matrix is optimal.
It is known that the Gross-Keating invariant $\GK(B)=\ua=(a_1, a_2, \ldots, a_n)$ is given by $a_i=\ord(t_i)$ for $i=1,2, \ldots, n$.
In particular,  $T$ is optimal.
For a proof, see Bouw \cite{Bouw} Proposition 2.6.
The proof of \cite{Bouw} is valid for any non-dyadic field.
Using this, one can easily show Theorem \ref{thm:0.1}.
Note that $\GK(B)=(0, \ldots, 0)$ if and only if $B\in\GL_n(\frko)$.

The ``if part" of Theorem \ref{thm:0.2} follows from Proposition \ref{prop:1.1}.
%To prove the ``only if part" of Theorem \ref{thm:0.2}, 
Let $\ua=(a_1, a_2, \ldots, a_n)$ be a non-decreasing sequence.
Suppose that $B$ is optimal and $\GK(B)=\ua$.
By Lemma \ref{lem:1.4}, we have  $\GK(B^{(n_1)})=(\underbrace{a_1, \ldots, a_1}_{n_1})$.
Then, we have $\vpi^{-a_1}B^{(n_1)}\in\GL_{n_1}(\frko)$.
It follows that there exists $U_1\in G_\ua$ such that $B[U_1]$ is of the form $B^{(n_1)}\perp B'$.
Repeating this argument, one can show that there exists $U_2\in G_\ua^\bigtriangleup$ such that $B[U_2]$ is a Jordan splitting of $B$.
Suppose that $B[U]$ is also optimal with $U\in\GL_n(\frko)$.
Then by the same argument as above, there exists $U_3\in G_\ua^\bigtriangleup$ such that $B[UU_3]$ is a Jordan splitting of $B$.
It is well-known that two Jordan splittings of $B$ are $\GL_{n_1}(\frko)\times \cdots\times \GL_{n_r}(\frko)$-equivalent. 
Moreover, if $T$ is a Jordan splitting of $B$, then $\{U\in\GL_n(\frko)\,|\, T[U]=T\}\subset G_\ua$.
Thus we obtain the ``only if part" of Theorem \ref{thm:0.2}.

Suppose that $B\in \calm(\ua)$. 
Then as we have seen above, $\GK(B)=\ua$ if and only if there exists $U\in G_\ua^\bigtriangleup$ such that $B=T[U]$, where $T$ is a  Jordan splitting of $B$ and $\GK(T)=\ua$.
Equivalently,  $\GK(B)=\ua$ if and only if $\ord(\det B^{(n^\ast_s)})=|\ua^{(n^\ast_s)}|$ for $s=1, \ldots, r$.
Theorem \ref{thm:0.3} follows from this.
Let $B$ and $B_1$ be as in Theorem \ref{thm:0.4}.
As we have seen above, there exists $U\in G_\ua^\bigtriangleup$ such that $B_1=B[U]$.
It follows that $B^{(k)}\sim B_1^{(k)}$.
Hence we obtain Theorem \ref{thm:0.4}.
\end{remark}

\section{Binary quadratic forms}
\label{sec:2}

Hereafter, until the end of section \ref{sec:5}, we assume that $F$ is dyadic.

Let $(L, Q)$ and $(L_1, Q_1)$ be quadratic modules of rank $n$ over $\frko$.
We say that $(L, Q)$ and $(L_1, Q_1)$ are weakly equivalent if there exists an isomorphism $\iot:L\rightarrow L_1$ and a unit $u\in \frko^\times$ such that $u Q_1(\iot(x))=Q(x)$ for any $x\in L$.
Similarly, we say that $B, B_1\in \calh_n(\frko)$ are weakly equivalent if there exists a unimodular matrix $U\in\GL_n(\frko)$ and a unit $u \in\frko^\times$ such that $u B_1=B[U]$.
If $B$ and $B_1$ are weakly equivalent, then $\GK(B)=\GK(B_1)$.

Recall that a half-integral symmetric matrix $B\in\calh_n(\frko)$ is primitive if $\vpi^{-1}B\notin \calh_n(\frko)$.
It is well-known that $B$ is primitive if and only if $\mathrm{n}(L)=\frko$, where $L$ is the  quadratic module associated to $B$.
Let $\GK(B)=(a_1, a_2, \ldots, a_n)$.
It is obvious that if $B$ is not primitive, then $a_1>0$.
Conversely, if $B$ is primitive, then $a_1=0$ by Lemma \ref{lem:1.2}.
Thus $B$ is primitive if and only if $a_1=0$.

We define the integer $e$ by $|2|^{-1}=q^e$.
Since we have assumed that $F$ is dyadic, $e$ is equal to the ramification index of $F/\QQ_2$.
It is well-known that $1+4\frkp\subset \frko^{\times 2}$.
For $\xi\in F^\times$, we denote the discriminant ideal of $F(\sqrt{\xi})/F$ by $\mathfrak{D}_\xi$.
Then $\ord(\mathfrak{D}_\xi)=0$ if and only if $\xi\in (1+4\frko)\frko^{\times 2}$ for $\xi\in\frko^\times$.
Moreover, $[(1+4\frko)\frko^{\times 2}:\frko^{\times 2}]=2$.
It is easy to see that if $\ord(\mathfrak{D}_{\xi})=0$, then $\ord(\mathfrak{D}_{\xi\xi'})=\ord(\mathfrak{D}_{\xi'})$ for any $\xi'\in F^\times$.
(See e.g., O'Meara \cite{omeara} \S63A.)

Let $E/F$ be a  semi-simple quadratic algebra.
This means that $E$ is a quadratic extension of $F$ or $E=F\oplus F$.
The non-trivial automorphism of $E/F$ is denoted by $x\mapsto \bar x$.
Note that if  $E=F\oplus F$, we have $\overline{(x_1, x_2)}=(x_2, x_1)$.
Let $\frko_E$ be the maximal order of $E$.
In the case $E=F\oplus F$, $\frko_E=\frko\oplus \frko$.
The discriminant ideal of $E/F$ is denoted by $\frkD_E$.
The order $\frko_{E, f}$ of conductor $f$ for $E/F$ is defined by $\frko_{E, f}=\frko+\frkp^f\frko_E$.
Any open $\frko$-subring of $\frko_E$ is of the form $\frko_{E, f}$ for some non-negative integer $f$.
We say that $E/F$ is unramified, if $E=F\oplus F$ or $E/F$ is an unramified quadratic extension.
Then $E/F$ is unramified if and only if $\ord(\frkD_E)=0$.

\begin{proposition} % Proposition 2.1
\label{prop:2.1}
Let $B\in\calhnd_2(\frko)$ be a primitive half-integral symmetric matrix of size $2$ and $(L, Q)$ its associated quadratic module.
Put $E=F(\sqrt{D_B})/F$.
When $D_B\in F^{\times 2}$, we understand $E=F\oplus F$.
Put $f=(\ord(D_B)-\ord(\mathfrak{D}_E))/2$.
Then $f$ is an integer and $(L, Q)$ is weakly equivalent to $(\frko_{E, f}, \mathrm{N})$, where, $\mathrm{N}$ is the norm form for $E/F$.
\end{proposition}

\begin{proof}
Since $B$ is primitive, there exists $x_0\in L$ such that $u=Q(x_0)$ is a unit.
By replacing $Q$ by $u^{-1}Q$, we may assume $Q(x_0)=1$.
Let $R$ be the even Clifford algebra of $(L, Q)$ over $\frko$ (See \cite{knus}).
Then $R\otimes_\frko F$ is the even Clifford algebra of $(L\otimes F, Q\otimes F)$, which is isomorphic to $E$.
Thus $R\simeq \frko_{E, f}$ for some $f\geq 0$.
By Lemma 2.2.1 of Chapter 5 of \cite{knus}, we have $(L, Q)\simeq (\frko_{E,f}, \mathrm{N})$.
Let $\{1, \ome\}$ be an $\frko$-basis of $\frko_E$.
Then $\{1, \vpi^f\ome\}$ is an $\frko$-basis of $\frko_{E, f}$.
From this, we have $\ord(D_B)=\ord(\mathfrak{D}_E)+2f$, since $(D_B)=((\vpi^f\ome-\vpi^f\bar\ome)^2)$ and $\mathfrak{D}_B=((\ome-\bar\ome)^2)$.
\end{proof}
Note that two primitive binary forms $B, B'\in\calhnd_2(\frko)$ are weakly equivalent if and only if $D_B/D_{B'}\in \frko^{\times 2}$ by Proposition \ref{prop:2.1}.

\begin{proposition} % Proposition 2.2
\label{prop:2.2}
The Gross-Keating invariant of the binary quadratic form $(L, Q)=(\frko_{E,f}, \mathrm{N})$ is given by
\[
\begin{cases}
(0, 2f) & \text{ if $E/F$ is unramified,}
\\
(0, 2f+1) & \text{ if $E/F$ is ramified.}
\end{cases}
\]
\end{proposition}
\begin{proof}
Let $(a_1, a_2)$ be the Gross-Keating invariant of $(L, Q)$.
Since $(L, Q)$ is primitive, we have $a_1=0$.

\medskip
\noindent {\textbf{Step 1.}} 
We first consider the case $f=0$.
In this case, $L=\frko_E$.

Assume that $E/F$ is unramified and $(\ome_1, \ome_2)$ is any ordered $\frko$-basis of  $\frko_E$.
Then we have $(\ome_1, \ome_2)_Q\in\frko^\times$, since $F$ is dyadic.
It follows that $S((\ome_1, \ome_2))=\{(0,0)\}$, and so $a_1=a_2=0$ in this case.

Next, we assume $E/F$ is ramified.
Let $\vpi_E$ be a prime element of $\frko_E$.
Then $(1, \vpi_E)$ is an ordered $\frko$-basis of $\frko_E$ and  $(0, 1)\in S((1, \vpi_E))$.
It follows that $a_2\geq 1$.
On the other hand, let $(\psi_1, \psi_2)$ be an optimal basis.
If $a_2\geq 2$, then the $\frko$-module generated by $\{\psi_1, \vpi^{-[a_2/2]}\psi_2\}$ is also a quadratic module over $\frko$.
This contradicts the fact that $\frko_E$ is a maximal quadratic module.
It follows that $a_2\leq 1$, and so $a_2=1$.

\medskip
\noindent {\textbf{Step 2.}} 
We assume $f>0$.
Let $(\psi_1, \psi_2)$ be an optimal basis of $L$.
The $\frko$-module generated by $\{\psi_1, \vpi^{-[a_2/2]}\psi_2\}$ is also a quadratic module over $\frko$. 
It follows that $\vpi^{-[a_2/2]}\psi_2\in \frko_E$.
Thus we have $[a_2/2]\leq f$, i.e., $a_2\leq 2f+1$.
On the other hand, let $(1, \ome)$ be an optimal basis of $\frko_E$.
Then $(1, \vpi^f\ome)$ is an ordered $\frko$-basis of $\frko_{E, f}$ and
\[
\begin{cases}
(0, 2f)\in S((1, \vpi^f\ome)) & \text{ if $E/F$ is unramified,}
\\
(0, 2f+1)\in S((1, \vpi^f\ome))  & \text{ if $E/F$ is ramified.}
\end{cases}
\]
In particular, $a_2=2f+1$, if $E/F$ is ramified.
Assume that $a_2=2f+1$ and $E/F$ is unramified.
If $(\psi_1, \psi_2)$ is an optimal basis, then $\vpi^{-f}\psi_2\in\frko_E$, and so $(\psi_1, \vpi^{-f}\psi_2)$ is an ordered $\frko$-basis of $\frko_E$.
Then $(0,1)\in S((\psi_1, \vpi^{-f}\psi_2))$, this contradicts  $\GK(\frko_E)=(0,0)$.
This proves $a_2=2f$ in this case.
\end{proof}
\begin{corollary} % Corollary 2.1
\label{cor:2.1}
If $B\in\calhnd_2(\frko)$, then $|\GK(B)|=\Delta(B)$.
\end{corollary}
\begin{proof}
We may assume $B$ is primitive.
If $B$ is primitive, then the corollary follows from Proposition \ref{prop:2.2}.
\end{proof}

Recall that an element $B\in\calh_n(\frko)$ is said to be decomposable if $B\sim B_1\perp B_2$, for some $B_1\in\calh_s(\frko)$, $B_2\in\calh_t(\frko)$, $s, t<n$.
$B$ is said to be indecomposable if it is not decomposable.
\begin{lemma} % Lemma 2.1
\label{lem:2.1}
Let $B\in\calhnd_2(\frko)$ be a primitive binary form.
Then $B$ is decomposable if and only if $D_B\in 4\frko$.
\end{lemma}
\begin{proof}
Suppose that $B\sim (b_1)\perp (b_2)$ is decomposable.
Then we have $D_B=4b_1b_2\in 4\frko$.
Conversely, suppose that $D_B \in 4\frko$.
Then $B$ is weakly isomorphic to $\begin{pmatrix} 1 & 0 \\ 0 & -D_B/4\end{pmatrix}$ by the remark after Proposition \ref{prop:2.1}.
\end{proof}
\begin{definition} % Definition 2.1
$K\in \calh_2(\frko)$ is a primitive unramified binary (quadratic) form if  the quadratic module associated to $K$ is isomorphic to $(\frko_E, \mathrm{N})$ for an unramified 
quadratic algebra $E$.
\end{definition}
Clearly, $B$ is a primitive unramified binary form if and only if $\Del(B)=0$.
By Proposition \ref{prop:2.2}, it is also equivalent to $\GK(B)=(0,0)$.
Note also that $B\in\calh_2(\frko)$ is weakly equivalent to a primitive unramified binary form, then $B$ itself is a primitive unramified binary form, since $\mathrm{N}(\frko_E^\times)=\frko^\times$.
A primitive unramified binary form is indecomposable by Lemma \ref{lem:2.1},  
If $\begin{pmatrix} a & b/2 \\ b/2 & c \end{pmatrix}\in\calh_2(\frko)$ is a primitive unramified binary form, then the proof of Proposition \ref{prop:2.2} shows  $b\in\frko^\times$.
Conversely, If $b\in \frko^\times$, then $\begin{pmatrix} a & b/2 \\ b/2 & c \end{pmatrix}\in\calh_2(\frko)$ is a primitive unramified binary form, since $\mathrm{ord}(\frkD_{b^2-4ac})=0$.
More precisely, a primitive unramified binary form is isomorphic to either
\[
H=\begin{pmatrix}
0 & 1/2 \\
1/2 & 0
\end{pmatrix}
\quad \text{ or }
\quad
Y=\begin{pmatrix}
1 & 1/2 \\
1/2 & c
\end{pmatrix},
\]
where $c\in\frko$ and $1-4c\notin \frko^{\times 2}$.
Note that $Y\perp Y\sim H\perp H$.

We characterize optimal binary forms as follows.
\begin{proposition} % Proposition 2.3
\label{prop:2.3}
Assume that $B=\begin{pmatrix} b_{11} & b_{12} \\ b_{12} & b_{22}\end{pmatrix}\in \calm((a_1, a_2))$ and $a_1\leq a_2$.
\begin{itemize}
\item[(1)]
If $a_1=a_2$, then
\[
\text{$\GK(B)=(a_1, a_2)$} \Longleftrightarrow  \mathrm{ord}(2 b_{12})=a_1.
\]
\item[(2)]
If $a_2-a_1=2f>0$, with $f\in \ZZ_{> 0}$, then
\[
\text{$\GK(B)=(a_1, a_2)$} \Longleftrightarrow  \mathrm{ord}(b_{11})=a_1, \mathrm{ord}(2 b_{12})=a_1+f.
\]
\item[(3)] 
If $a_2-a_1=2f+1$, with $f\in \ZZ_{\geq 0}$, then
\[
\text{$\GK(B)=(a_1, a_2)$} \Longleftrightarrow  \mathrm{ord}(b_{11})=a_1, \mathrm{ord}(b_{22})=a_2.
\]
\end{itemize}
\end{proposition}
\begin{proof}
By replacing $B$ by $\vpi^{-a_1}B$, we may assume $a_1=0$.
We have already seen (1).
We prove (2).
Since $B$ is primitive, we have $\mathrm{ord}(b_{11})=0$.
Put $B'=B[\mathrm{diag}(1, \vpi^{-f})]\in\calh_2(\frko)$.
Then $B'$ is a primitive unramified binary form if and only if $\mathrm{ord}(2b_{12})=f$.
On the other hand, $\GK(B)=2f$ if and only if $\Del(B)=2f$ by Corollary \ref{cor:2.1}.
Since $\Del(B)=\Del(B')+2f$, we have (2).
Now we prove (3).
Since $B$ is primitive, we have $\mathrm{ord}(b_{11})=0$. 
Put $B'=B[\mathrm{diag}(1, \vpi^{-f})]\in\calh_2(\frko)$.
As in the proof of (2), $\GK(B)=(0, 2f+1)$ if and only if $\GK(B')=(0,1)$.
Thus we may assume $f=0$.
If $\GK(B)=(0,1)$, then clearly we have $\ord(2b_{12})>1$ and $\ord(b_{22})=1$.
Conversely, suppose that
\[
\ord(b_{11})=0, \quad \ord(2b_{12})>0, \quad \ord(b_{22})=1.
\]
Then $b_{11}x^2+2b_{12}x+b_{22}$ is an Eisenstein polynomial.
The roots of this polynomial are prime elements of the ramified quadratic extension $F(\sqrt{D_B})/F$.
In particular, the quadratic module associated to $B$ is weakly isomorphic to the maximal order of $F(\sqrt{D_B})$.
By Proposition \ref{prop:2.2}, we have $\GK(B)=(0,1)$.
\end{proof}

\noindent \textbf{Example}.
Suppose that $F=\QQ_2$.
Then $B=\begin{pmatrix} 1 & 0 \\ 0 & 1 \end{pmatrix}$ is not optimal.
It is equivalent to an optimal form $\begin{pmatrix} 1 & 1 \\ 1 & 2 \end{pmatrix}$.
Note that $\GK(B)=(0,1)$.

\begin{lemma} % Lemma 2.2
\label{lem:2.2}
Suppose that $B$ and $B'$ are primitive unramified binary forms.
If $B-B'\in\vpi\calh_2(\frko)$, then  $\xi_B=\xi_{B'}$.
\end{lemma}
\begin{proof}
One can easily show that $D_B-D_{B'}\in 4\frkp$,  and so $\xi_B=\xi_{B'}$.
\end{proof}

\section{Reduced forms}
\label{sec:3}

In this section, we introduce reduced forms in a somewhat generalized way.
We do \emph{not} assume $\ua=(a_1, \ldots, a_n)\in\ZZn$ is non-decreasing, unless otherwise  stated.

The $s$-th block $I_s\subset \{1, 2, \ldots, n\}$ is given by 
\begin{align*}
I_1&=\{i\,|\, a_i=\min(a_1, a_2, \ldots, a_n)\}, \\
I_2&=\{i\,|\, a_i=\min\{ a_j\,|\, j\notin I_1\}\}, \\
I_3&=\{i\,|\, a_i=\min\{ a_j\,|\, j\notin I_1\cup I_2\}\}, \\
&\cdots \\
I_r&=\{i\,|\, a_i=\max\{a_j\,|\, 1\leq j\leq n\}\}.
\end{align*}
Let $\frkS_n$ be the symmetric group of degree $n$.
Recall that a permutation $\sig\in \frkS_n$ is an involution if $\sig^2=\mathrm{id}$.
For an involution $\sig$, we set
\begin{align*}
\calp^0&=\calp^0(\sig)=\{i\,| 1\leq i\leq n,\;  i=\sig(i)\}, \\
\calp^+&=\calp^+(\sig)=\{i\,| 1\leq i\leq n,\;  a_i>a_{\sig(i)}\}, \\
\calp^-&=\calp^-(\sig)=\{i\,| 1\leq i\leq n,\;  a_i<a_{\sig(i)}\}.
\end{align*}
For $s=1,2, \ldots, r$, put
\begin{align*}
\calp_s^0=&\calp_s^0(\sig)=\calp^0\cap I_s, \\
\calp_s^+=&\calp_s^+(\sig)=\calp^+\cap I_s, \\
\calp_s^-=&\calp_s^-(\sig)=\calp^-\cap I_s.
\end{align*}
\begin{definition} % Definition 3.1
We shall say that an involution $\sig\in\frkS_n$ is $\ua$-admissible if the following three conditions are satisfied.
\begin{itemize}
\item[(i)] 
$\calp^0$ has at most two elements.
If $\calp^0$ has two distinct elements $i$ and $j$, then $a_i\not\equiv a_j$ mod $2$.
Moreover, if $i\in \calp^0$, then
\[
a_i=\max\{ a_j \, |\, j\in \calp^0\cup\calp^+, \, a_j\equiv a_i \text{ mod }2\}.
\] 
\item[(ii)]
For $s=1, 2, \ldots, r$, we have
\[
\sharp\calp_s^+\leq 1, \quad \sharp \calp_s^-+\sharp \calp_s^0\leq 1.
\]
\item[(iii)]
If $i\in\calp^-$, then 
\[
a_{\sig(i)}=\min\{a_j\,|\,j\in\calp^+,\, a_j>a_i, \ a_j\equiv a_i \text{ mod }2\}.
\]
Similarly, if $i\in\calp^+$, then 
\[
a_{\sig(i)}=\max\{a_j\,|\,j\in\calp^-,\, a_j<a_i, \ a_j\equiv a_i \text{ mod }2\}.
\]
\end{itemize}
\end{definition}
It is easy to see that such an $\ua$-admissible involution exists for any $\ua\in\ZZn$.
When $\sig$ is an $\ua$-admissible involution, we call the pair $(\ua,\sig)$ a GK type.
Note that $\sharp \calp^0(\sig)=1$ if $n$ is odd.
If $n$ is even, then $\sharp\calp^0(\sig)=0$ or $\sharp\calp^0(\sig)=2$ according as $|\ua|$ is even or odd.

\begin{definition} % Definition 3.2
\label{def:3.2}
Let $\sig\in\frkS_n$ be an $\ua$-admissible involution.
We say that $B=(b_{ij})\in \calm(\ua)$ is a reduced form of (generalized) GK type $(\ua, \sig)$ if the following conditions are satisfied.
\begin{itemize}
\item[(1)]  If $i\notin\calp^0$, $j=\sig(i)$, and $a_i \leq a_j$, then 
\[
\GK\left(\begin{pmatrix} b_{ii} & b_{ij} \\ b_{ij} & b_{jj} \end{pmatrix}\right)=(a_i, a_j).
\]
Note that this condition is equivalent to the following condition.
\[
 \left\{
 \begin{array}{ll}
\ord(2 b_{ij})={\ds \frac{a_i+a_j}2}, \quad i\notin\calp^0,\; j=\sig(i) \\ 
\noalign{\vskip 6pt}
\ord(b_{ii})=a_i, \quad i\in\calp^-.
 \end{array}
 \right.
\]
\item[(2)] If $i\in\calp^0$, then
\[
\ord(b_{ii})=a_i.
\]
\item[(3)] If $j\neq i, \sig(i)$, then
\[
\ord(2 b_{ij})>\frac{a_i+a_j}2,
\]
\end{itemize}
\end{definition}
We often say $B$ is a reduced form of GK type $\ua$ without mentioning $\sig$.
We formally think of the empty matrix as a reduced form of GK type $\emptyset$.

\begin{lemma} % Lemma 3.1
\label{lem:3.1}
Suppose that $B\in\calh_n(\frko)$ is a reduced form of GK type $\ua=(\underbrace{0, \ldots, 0}_n)$.
\begin{itemize}
\item[(1)] If $n=2m$ is even, then  $B\sim K_1\perp\cdots \perp K_m$, where $K_1, \ldots, K_m$ are primitive unramified binary forms.
\item[(2)] If $n=2m+1$ is odd, then $B\sim (u)\perp K_1\perp \cdots \perp K_m$, where $K_1,\ldots, K_m$ are primitive unramified binary forms and $u\in\frko^\times$.
\end{itemize}
\end{lemma}
\begin{proof}
We prove this lemma by induction with respect to $n$.
For $n\leq 2$, the lemma is obvious.
Assume that $n>2$.
We may assume $\sig(1)=2$ by changing the coordinates.
Then $B$ is of the form
\[
B=\begin{pmatrix} K & X \\ {}^t\!X & B_{22}\end{pmatrix}
\]
such that $2X\in \vpi\mathrm{M}_{2, n-2}(\frko)$ and $B_{22}$ is a reduced form of GK type  $(\underbrace{0, \ldots, 0}_{n-2})$.
Since $K^{-1}X\in \vpi\mathrm{M}_{2, n-2}(\frko)$, $B$ is equivalent to
\[
B\left[
\begin{pmatrix}
1 & -K^{-1} X \\
0 & 1
\end{pmatrix}
\right]
=
\begin{pmatrix}
K & 0 \\
0 & B_{22}-K^{-1}[X]
\end{pmatrix}.
\]
Then $B_{22}-K^{-1}[X]$ is a reduced form, since $K^{-1}[X]\in\vpi^2\calh_{n-2}(\frko)$.
Hence the lemma.
\end{proof}

\begin{proposition} % Proposition 3.1
\label{prop:3.1}
Suppose that $B\in\calh_n(\frko)$ is a reduced form of GK type $\ua=(\underbrace{0, 0, \ldots, 0}_{s}, \underbrace{1, 1, \ldots, 1}_{t})$.
Then there exists $U\in G_\ua$ such that
\[
B[U]=K_1\perp K_2\perp \cdots  \perp K_{[s/2]} \perp B' \perp \vpi (K'_1\perp  K'_2\perp \cdots \perp K_{[t/2]} ),
\]
where $K_1, K_2, \ldots K_{[s/2]}$ and $K'_1, K'_2, \ldots K'_{[t/2]}$ are primitive unramified binary forms and $B'$ is a reduced form of GK type $\emptyset, (0), (1)$, or $(0,1)$.
\end{proposition}
\begin{proof}
First assume $s\geq 2$.
We may assume $\sig(1)=2$.
Then $B$ is of the form
\[
B=\begin{pmatrix} K & X \\ {}^t\!X & B_{22}\end{pmatrix},
\]
where $K$ is a primitive unramified binary form.
Moreover, $2X\in \vpi\mathrm{M}_{2, n-2}(\frko)$ and $B_{22}$ is a reduced form of GK type  $(\underbrace{0, \ldots, 0}_{s-2}, \underbrace{1, \ldots, 1}_t)$.
Put
\[
U=\begin{pmatrix}
1 & -K^{-1} X \\
0 & 1
\end{pmatrix}.
\]
Since $K^{-1}X\in \mathrm{M}_{2, n-2}(\frko)$, we have $U\in G_\ua$.
Then we have
\[
B[U]=
\begin{pmatrix}
K & 0 \\
0 & B_{22}-K^{-1}[X]
\end{pmatrix}.
\]
Since $K^{-1}[X]\in\vpi^2\calh_{n-2}(\frko)$, $B_{22}-K^{-1}[X]$ is also a reduced form of  GK type $(\underbrace{0, \ldots, 0}_{s-2}, \underbrace{1, \ldots, 1}_t)$.
Thus we may assume $s\leq 1$.
In particular, the proposition is proved if $t=0$.
The case $s=0$ is reduced to the case $t=0$ by considering $\vpi^{-1}B$.
Thus we may assume $s=1$ and $t\geq 2$.
We may assume $\sig(n)=n-1$.
Write $B$ in a block form as follows.
\[
\begin{array}{ccccc}
&\hskip -15pt
\overbrace{\hphantom{B_{11}}}^{1} \;
\overbrace{\hphantom{ B_{12}}}^{t-2}  \;
\overbrace{\hphantom{ B_{15}}}^{2} \\ & 
B=\left(
\begin{array}{cccl}
B_{11} & B_{12} & B_{13}  \\
{}^t\! B_{12} & B_{22} & B_{23}  \\
{}^t\! B_{13} & {}^t\! B_{23} & B_{33}  \\
\end{array} \hskip -0pt
\right) \hskip -5pt
\begin{array}{l}
 \left.\vphantom{B_{11}} \right\} \text{\footnotesize${1}$} \\
 \left.\vphantom{B_{11}} \right\} \text{\footnotesize${t-2}$} \\
 \left.\vphantom{B_{11}} \right\} \text{\footnotesize${2}$.} 
\end{array}
\end{array}
\]
Then we have $2B_{13}\in \vpi \mathrm{M}_{1,2}(\frko)$, $2B_{23}\in \vpi^2 \mathrm{M}_{t-2, 2}(\frko)$.
Moreover, $\vpi^{-1}B_{33}$ is an unramified primitive binary form.
Put $X_1=-B_{33}^{-1}\cdot{}^t\! B_{13}$ and $X_2=-B_{33}^{-1}\cdot{}^t\! B_{23}$.
Then we have
\[
X_1\in \mathrm{M}_{2,1}(\frko), \quad
X_2\in \vpi\mathrm{M}_{2,t-2}(\frko).
\]
Put
\begin{align*}
B
\left[
\begin{pmatrix}
1 & 0 & 0 \\
0 & \mathbf{1}_{t-2} & 0 \\
X_1 & X_2 & \mathbf{1}_2
\end{pmatrix}
\right]
=&\begin{pmatrix}
B'_{11} & B'_{12} & 0   \\
{}^t\! B'_{12} & B'_{22} & 0   \\
0 & 0 & B_{33}
\end{pmatrix},
\\
\begin{pmatrix} B'_{11} & B'_{12} \\ {}^t\!B'_{12} & B'_{22}\end{pmatrix}=&
\begin{pmatrix} B_{11} & B_{12} \\ {}^t\!B_{12} & B_{22}\end{pmatrix}
- B_{33}[\begin{pmatrix}X_1\, X_2\end{pmatrix}].
\end{align*}
Then we have
\[
\begin{pmatrix} B_{11} & B_{12} \\ {}^t\!B_{12} & B_{22}\end{pmatrix}
-
\begin{pmatrix} B'_{11} & B'_{12} \\ {}^t\!B'_{12} & B'_{22}\end{pmatrix}
=B_{33}[\begin{pmatrix}X_1\, X_2\end{pmatrix}]\in \calm^0(0, \underbrace{1,\dots,1}_{t-2}).
\]
It follows that
$\begin{pmatrix} B'_{11} & B'_{12} \\ {}^t\!B'_{12} & B'_{22}\end{pmatrix}$ is a reduced form of GK type $(0, \underbrace{1, 1, \ldots, 1}_{t-2})$.
Repeating this argument, the lemma is reduced to the case $t\leq 1$.
Hence the lemma is proved.
\end{proof}

\begin{lemma} % Lemma 3.2
\label{lem:3.2}
Suppose that $B\in\calh_n(\frko)$ is a reduced form of GK type $\ua=(\underbrace{0, \ldots, 0}_s, \underbrace{1, \ldots, 1}_t)$.
Then we have $\Del(B)=|\ua|$.
\end{lemma}
\begin{proof}
Note that $\Del(B\perp \vpi^c K)=\Del(B)+2c$ if $K$ is a primitive unramified binary form.
By Proposition \ref{prop:3.1}, it is enough to consider the case $s, t\leq 1$.
The case $s=0$ or $t=0$ is trivial.
The case $s=t=1$ follows from Corollary \ref{cor:2.1} and Proposition \ref{prop:2.3}.
\end{proof}

\begin{proposition} % Proposition 3.2
\label{prop:3.2}
Suppose that $B\in\calh_n(\frko)$ is a reduced form of GK type  $\ua=(a_1, a_2, \ldots, a_n)$.
Then we have
\[
|\ua|=\Del(B).
\]
\end{proposition}
\begin{proof}
Put 
\[
B'=B[\mathrm{diag}(\vpi^{-[a_1/2]}, \vpi^{-[a_2/2]}, \ldots, \vpi^{-[a_n/2]})].
\]
Then $B'$ is a reduced form of GK type $\ua'=(a'_1, a'_2, \ldots, a'_n)$, where $a'_i=a_i-2[a_i/2]$.
Since
\begin{align*}
|\ua|&=|\ua'| + 2\sum_{i=1}^n \left[\frac{a_i}2\right], \\
\Delta(B)&=\Delta(B') + 2\sum_{i=1}^n \left[\frac{a_i}2\right],
\end{align*}
it is enough to consider the case $a_1, a_2, \ldots, a_n\leq 1$.
By changing the coordinate, we may assume $\ua=(\underbrace{0, \ldots, 0}_s, \underbrace{1, \ldots, 1}_t)$.
In this case, the proposition follows from Lemma \ref{lem:3.2}.
\end{proof}

Let $\ua=(a_1, a_2, \ldots, a_n)$ be a sequence of integers whose components are allowed to be negative.
For such a sequence, we put $\calm(\ua)=\vpi^{-a_0}\calm(a_0+\ua)$, where $a_0$ is a sufficiently large integer and $a_0+\ua=(a_0+a_1, a_0+a_2, \ldots, a_0+a_n)$.
Obviously, this definition does not depend on a choice of $a_0$.
Similarly, we say that $B\in\calh_n(\frko)$ is a reduced form of GK type $(\ua, \sig)$, if $\vpi^{a_0}B$ is a reduced form of GK type $(a_0+\ua, \sig)$.

\begin{lemma} %Lemma 3.3
\label{lem:3.2a}
Suppose that $B\in\calh_n(\frko)$ is a reduced form of GK type $(\ua, \sig)$.
If $\calp^0(\sig)=\emptyset$, then $(4B)^{-1}\in\calm(-\ua)$.
\end{lemma}
\begin{proof}
We first note that the lemma holds for $\ua=(0,0)$.
In fact, $B$ is a primitive unramified binary form in this case.
Then $(4B)^{-1}$ is also a primitive unramified binary form.

Now we consider general case.
As is the proof of Proposition \ref{prop:3.2}, we may assume that $\ua=(\underbrace{0, \ldots, 0}_s, \underbrace{1, \ldots, 1}_t)$.
The assumption $\calp^0(\sig)=\emptyset$ implies that both $s$ and $t$ are even.
By the proof of Proposition \ref{prop:3.1}, there exist $K_1\in\calh_s(\frko)$, $K_2\in\calh_t(\frko)$ and $X\in\vpi \mathrm{M}_{s, t}(\frko)$ such that the following conditions hold:
\begin{itemize}
\item[(1)] $K_1$ and $K_2$ are equivalent to direct sums of primitive unramified binary forms.
\item[(2)] We have
\[
B=
\begin{pmatrix}
K_1 & 0 \\ 0 & \vpi K_2
\end{pmatrix}\left[
\begin{pmatrix}
1 & X \\ 0 & 1 \end{pmatrix}
\right].
\]
\end{itemize}
Then we have
\[
\GK((4K_1)^{-1})=(\underbrace{0, \ldots, 0}_s), \quad \GK((4K_2)^{-1})=(\underbrace{0, \ldots, 0}_t).
\]
It follows that
\[
(4B)^{-1}=
\begin{pmatrix}
(4K_1)^{-1} & 0 \\ 0 & \vpi^{-1} (4K_2)^{-1}
\end{pmatrix}\left[
\begin{pmatrix}
1 & 0 \\ -\,{}^t\! X & 1 \end{pmatrix}
\right]\in 
\calm(-\ua).
\]
Hence we have proved the lemma.
\end{proof}

For $B\in\calhnd_n(\frko)$, let $\eta_B$ be the Clifford invariant of $B$ introduced in Definition \ref{def:0.4}.
By \cite{lam}, Chapter 5, section 3, (3.13), we have
\[
\eta_{B_1\bot B_2}=
\begin{cases}
\eta_{B_1}\eta_{B_2}\langle D_{B_1}, D_{B_2}\rangle
& \text{ if $n_1\equiv n_2$ mod $2$},
\\
\eta_{B_1}\eta_{B_2}\langle D_{B_1}, -D_{B_2}\rangle
& \text{ if $n_1$ is even and $n_2$ is odd}
\end{cases}
\]
for $B_1\in\calh^\mathrm{nd}_{n_1}(\frko)$ and  $B_2\in\calh^\mathrm{nd}_{n_2}(\frko)$.

\begin{lemma} %Lemma 3.4
\label{lem:3.3}
For $B\in\calhnd_n(\frko)$ such that $B^{(n-1)}\in\calhnd_{n-1}(\frko)$, then we have
\[
\eta_B=\eta_{B^{(n-1)}} \langle D_B, D_{B^{(n-1)}} \rangle.
\]
\end{lemma}
\begin{proof}
Note that $B$ is $\GL_n(F)$-equivalent to  $B^{(n-1)}\perp ((-1)^{n-1}D_B D_{B^{(n-1)}})$.
Assume that $n$ is odd.
Then we have 
\[
\eta_B=\eta_{B^{(n-1)}}\langle D_B, -D_B D_{B^{(n-1)}}\rangle=\eta_{B^{(n-1)}}\langle D_B, D_{B^{(n-1)}}\rangle.
\]
The case when $n$ is even is similar.
\end{proof}

\begin{lemma} % Lemma 3.5
\label{lem:3.4}
Let $K$ be a primitive unramified binary form.
For $B\in\calhnd_n(\frko)$, we have
\[
\eta_{B\perp \vpi^a K}=\eta_B \xi_K^{a+\ord(D_B)}.
\]
\end{lemma}
\begin{proof}
Note that $\eta_{\vpi^a K}=\xi_K^a$.
Hence we have
\[
\eta_{B\perp \vpi^a K}=\eta_B \eta_{\vpi^a K} \langle D_B, D_K \rangle
=\eta_B \xi_K^{a+\ord(D_B)}.
\]
\end{proof}

\begin{lemma} %Lemma 3.6
\label{lem:3.5} 
Let $B \in \calh_n(\frko)$ be a half-integral symmetric matrix with $\GK(B)=\ua=(a_1,\ldots, a_n)$.
Assume that $a_1=\cdots=a_n$. 
Then we have
\[
\eta_B=
\begin{cases}
1 & \text{ if $n$ is odd,} \\
\xi_B^{a_1} & \text{ if $n$ is even.}
\end{cases}
\]
\end{lemma}
\begin{proof}
By Theorem \ref{thm:4.1} and Lemma \ref{lem:3.1}, we may assume
\[
B=
\begin{cases}
\vpi^{a_1}((u)\perp K_1\perp \cdots \perp K_{[n/2]}) & \text{ if $n$ is odd,} \\
\vpi^{a_1}(K_1\perp \cdots \perp K_{n/2}) & \text{ if $n$ is even,}
\end{cases}
\]
where $u\in\frko^\times$ and $K_1, \ldots, K_{[n/2]}$ are primitive unramified binary forms.
Then the lemma follows from Lemma \ref{lem:3.4}.
\end{proof}

\begin{proposition} % Proposition 3.3
\label{prop:3.3}
Suppose that $B, T\in\calh_n(\frko)$ are reduced forms of GK type  $\ua$.
If $B-T\in\calm^0(\ua)$, then the following assertions  (a) and  (b) hold.
\begin{itemize}
\item[(a)] If $n$ is even, then $\xi_B=\xi_{T}$.
\item[(b)] If $n$ odd, then $\eta_B=\eta_{T}$.
\item[(c)] If $n$ is even and $\xi_B\neq 0$, then $\eta_B=\eta_T$.
\end{itemize}
\end{proposition}
\begin{proof}
Put 
\begin{align*}
B'&=B[\mathrm{diag}(\vpi^{-[a_1/2]}, \vpi^{-[a_2/2]}, \ldots, \vpi^{-[a_n/2]})], \\
T'&=T[\mathrm{diag}(\vpi^{-[a_1/2]}, \vpi^{-[a_2/2]}, \ldots, \vpi^{-[a_n/2]})].
\end{align*}
Then $B'$ and $T'$  are reduced forms of GK type $\ua'=(a'_1, a'_2, \ldots, a'_n)$  and $B'-T'\in\calm^0(\ua')$, where $a'_i=a_i-2[a_i/2]$.
Thus we may assume $\ua=(\underbrace{0, \ldots, 0}_s, \underbrace{1, \ldots, 1}_t)$.
We first prove (a).
If both $s$ and $t$ are odd, then $\xi_B=\xi'_{B'}=0$.
We assume both $s$ and $t$ are even.
Suppose that $s\geq 2$.
We may assume $\sig(1)=2$.
Write $B$ and $T$ in block forms
\[
B=\begin{pmatrix}
B_{11} & B_{12} \\
{}^t\! B_{12} & B_{22}
\end{pmatrix},\quad 2B_{12}\in \vpi\mathrm{M}_{2, n-2}(\frko)
\]
and
\[
T=\begin{pmatrix}
T_{11} & T_{12} \\
{}^t T_{12} & B'_{22}
\end{pmatrix}, \quad 2T_{12}\in \vpi\mathrm{M}_{2, n-2}(\frko).
\]
Then $B_{11}$ and $T_{11}$ are unramified primitive binary forms and $\xi_{B_{11}}=\xi_{T_{11}}$ by Lemma \ref{lem:2.2}.
Put
\[
B
\left[
\begin{pmatrix}
1 & -B_{11}^{-1}B_{12} \\
0 & 1
\end{pmatrix}
\right]
=\begin{pmatrix}
B_{11} & 0        \\
0 & B'
\end{pmatrix},\qquad
B'=B_{22}- B_{11}^{-1}[B_{12}].
\]
Then we have $\xi_B=\xi_{B_{11}}\xi_{B'}$.
Similarly, put
\[
T
\left[
\begin{pmatrix}
1 & -T_{11}^{-1}T_{12} \\
0 & 1
\end{pmatrix}
\right]
=\begin{pmatrix}
T_{11} & 0        \\
0 & T'
\end{pmatrix},\qquad
T'=T_{22}- T_{11}^{-1}[T_{12}].
\]
Then we have $\xi_T=\xi_{T_{11}}\xi_{T'}$.
Note that 
\[
B_{11}^{-1}[B_{12}], T_{11}^{-1}[T_{12}]\in\vpi^2\calh_{n-2}(\frko).
\]
It follows that $B'$ and $T'$ are reduced form of GK type $(\underbrace{0, \ldots, 0}_{s-2}, \underbrace{1, \ldots, 1}_t)$ and $B'-T'\in \calm^0(\underbrace{0, \ldots, 0}_{s-2}, \underbrace{1, \ldots, 1}_t)$.
Thus the proof is reduced to the case $s=0$.
The case $s=0$ is reduced to the case $t=0$, by replacing  $B$ and $T$ by $\vpi^{-1}B$ and $\vpi^{-1}T$, respectively.
Thus we have proved (a).

Next, we show (b).
By the same argument as above, the proof is reduced to the case $s\leq 1$ by using Lemma \ref{lem:3.4}.
If $s=0$, then $\eta_B=\eta_T=1$ by Lemma \ref{lem:3.5}.
Assume now $s=1$.
In this case $t$ is even.
Since the case $t=0$ is trivial, we may assume $t\geq 2$.
We may assume $\sig(n)=n-1$.
Write $B$ and $T$ in block forms as follows.
\[
\begin{array}{ccccc}
&\hskip -15pt
\overbrace{\hphantom{B_{11}}}^{1} \;
\overbrace{\hphantom{ B_{12}}}^{t-2}  \;
\overbrace{\hphantom{ B_{15}}}^{2} \\ & 
B=\left(
\begin{array}{cccl}
B_{11} & B_{12} & B_{13}  \\
{}^t\! B_{12} & B_{22} & B_{23}  \\
{}^t\! B_{13} & {}^t\! B_{23} & B_{33}  \\
\end{array} \hskip -0pt
\right) \hskip -5pt
\begin{array}{l}
 \left.\vphantom{B_{11}} \right\} \text{\footnotesize${1}$} \\
 \left.\vphantom{B_{11}} \right\} \text{\footnotesize${t-2}$} \\
 \left.\vphantom{B_{11}} \right\} \text{\footnotesize${2}$,} 
\end{array}
\end{array}
\]
\[
\begin{array}{ccccc}
&\hskip -15pt
\overbrace{\hphantom{T_{11}}}^{1} \;
\overbrace{\hphantom{T_{12}}}^{t-2}  \;
\overbrace{\hphantom{T_{15}}}^{2} \\ & 
T=\left(
\begin{array}{cccl}
T_{11} & T_{12} & T_{13}  \\
{}^t T_{12} & T_{22} & T_{23}  \\
{}^t T_{13} & {}^t T_{23} & T_{33}  \\
\end{array} \hskip -0pt
\right) \hskip -5pt
\begin{array}{l}
 \left.\vphantom{T_{11}} \right\} \text{\footnotesize${1}$} \\
 \left.\vphantom{T_{11}} \right\} \text{\footnotesize${t-2}$} \\
 \left.\vphantom{T_{11}} \right\} \text{\footnotesize${2}$.} 
\end{array}
\end{array}
\]
Put 
\begin{align*}
X_1&=-B_{33}^{-1}\cdot{}^t\! B_{13}, \quad  X_2=-B_{33}^{-1}\cdot{}^t\! B_{23}, \\
Y_1&=-T_{33}^{-1}\cdot{}^t T_{13}, \quad Y_2=-T_{33}^{-1}\cdot{}^t T_{23}.
\end{align*}
\begin{align*}
B
\left[
\begin{pmatrix}
1 & 0 & 0 \\
0 & \mathbf{1}_{t-2} & 0 \\
X_1 & X_2 & \mathbf{1}_2
\end{pmatrix}
\right]
=&\begin{pmatrix}
B'_{11} & B'_{12} & 0   \\
{}^t\! B'_{12} & B'_{22} & 0   \\
0 & 0 & B_{33}
\end{pmatrix},
\\
T
\left[
\begin{pmatrix}
1 & 0 & 0 \\
0 & \mathbf{1}_{t-2} & 0 \\
Y_1 & Y_2 & \mathbf{1}_2
\end{pmatrix}
\right]
=&\begin{pmatrix}
T'_{11} & T'_{12} & 0        \\
{}^t T'_{12} & T'_{22} & 0        \\
0 & 0 & T_{33}
\end{pmatrix}.
\end{align*}
As in the proof of Proposition \ref{prop:3.1}, we have
\[
\begin{pmatrix} B_{11} & B_{12} \\ {}^t\!B_{12} & B_{22}\end{pmatrix}
-
\begin{pmatrix} B'_{11} & B'_{12} \\ {}^t\!B'_{12} & B'_{22}\end{pmatrix}
\in \calm^0(0, \underbrace{1,\dots,1}_{t-2}),
\]
\[
\begin{pmatrix} T_{11} & T_{12} \\ {}^t T_{12} & T_{22}\end{pmatrix}
-
\begin{pmatrix} T'_{11} & T'_{12} \\ {}^t T'_{12} & T'_{22}\end{pmatrix}
\in \calm^0(0, \underbrace{1,\dots,1}_{t-2}).
\]
By Lemma \ref{lem:2.2}, we have $\xi_{B_{33}}=\xi_{T_{33}}$.
On the other hand, by induction hypothesis, we have $\eta_{B'}=\eta_{T'}$, where
\[
B'=
\begin{pmatrix} B'_{11} & B'_{12} \\ {}^t\!B'_{12} & B'_{22}\end{pmatrix}, \quad 
T'=\begin{pmatrix} T'_{11} & T'_{12} \\ {}^t T'_{12} & T'_{22}\end{pmatrix}.
\]
By Lemma \ref{lem:3.4}, we have $\eta_B=\xi_{B_{33}}\eta_{B'}=\xi_{T_{33}}\eta_{T'}=\eta_T$.

Now we prove (c).
As in the previous cases, we may assume $\ua=(\underbrace{0, \ldots, 0}_s, \underbrace{1, \ldots, 1}_t)$.
Since $\xi_B\neq 0$, both $s$ and $t$ are even.
We proceed by induction with respect to $s$.
The case $s=0$ follows from (a) and Lemma \ref{lem:3.5}.
Suppose that $s\geq 2$.
As in the proof of (a), we can show
\begin{align*}
B\sim B_{11}\perp B', \quad \GK(B_{11})=(0,0),\; \GK(B')=(\underbrace{0, \ldots, 0}_{s-2}, \underbrace{1, \ldots, 1}_t), \\
T\sim T_{11}\perp T', \quad \GK(T_{11})=(0,0),\; \GK(T')=(\underbrace{0, \ldots, 0}_{s-2}, \underbrace{1, \ldots, 1}_t),
\end{align*}
where 
\[
B_{11}-T_{11}\in\calm^0((0,0)), \quad B'-T'\in \calm^0((\underbrace{0, \ldots, 0}_{s-2}, \underbrace{1, \ldots, 1}_t)).
\]
By (a) and Lemma \ref{lem:3.4}, we have $\eta_B=\eta_{B'}=\eta_{T'}=\eta_T$, as desired.
\end{proof}
\begin{remark} % Remark 3.1
If $n$ is even and $\xi_B\neq 0$, then (a) and (c) imply that $B$ and $T$ are $\GL_n(F)$-equivalent, but not $\GL_n(\frko)$-equivalent in general.
In the case $\xi_B=0$, $B$ and $T$ may not be $\GL_n(F)$-equivalent.
For example, put $B=\begin{pmatrix} -1 & 0 \\ 0 & 2 \end{pmatrix}$ and $T=\begin{pmatrix} -1 & 1 \\ 1 & -2 \end{pmatrix}$.
Then $\GK(B)=\GK(T)=(0,1)$ and $B-T\in\calm^0((0,1))$.
By easy calculation, $\eta_B=1$ and $\eta_T=-1$.
Note that $D_B=2$ and $D_T=-1$, and so the discriminant fields are different.
\end{remark}

\begin{proposition} % Proposition 3.4
\label{prop:3.4}
Suppose that $B\in\calh_n(\frko)$ and  $\GK(B)=(0, 0, \ldots, 0)$.
\begin{itemize}
\item[(1)] If $n=2m$ is even, then  $B\sim K_1\perp\cdots \perp K_m$, where $K_1, \ldots, K_m$ are primitive unramified binary forms.
\item[(2)] If $n=2m+1$ is odd, then $B\sim (u)\perp K_1\perp \cdots \perp K_m$, where $K_1,\ldots, K_m$ are primitive unramified binary forms and $u\in\frko^\times$.
\end{itemize}
\end{proposition}
\begin{proof}
By Proposition \ref{prop:1.1}, any half-integral symmetric matrix equivalent to $B$ is optimal, since $G_\ua=\GL_n(\frko)$ for $\ua=(0,\dots, 0)$.
It is well-known that $B$ is isomorphic to a direct sum of matrices of size $1$ or $2$.
By Lemma \ref{lem:1.4}, the Gross-Keating invariant of any direct summand is of the form $(0, 0, \ldots, 0)$.
Thus $B$ is isomorphic to 
\[
(u_1)\perp (u_2)\perp\cdots\perp (u_r)\perp K_1\perp K_2\perp\cdots\perp K_s,
\]
where $u_1, u_2, \ldots, u_r$ are units and $K_1, K_2, \ldots, K_s$ are primitive unramified binary forms.
Note that $\GK((u_1)\perp (u_2))\neq (0,0)$, since it is not a primitive unramified binary form.
Thus $B$ cannot contain a direct summand of the form $(u_1)\perp(u_2)$.
This shows that $r\leq 1$.
\end{proof}

\begin{lemma} %Lemma 3.7
\label{lem:3.6}
If $B \in\calm_n(\ua)$, then we have
\[
|\ua|\leq \Del(B).
\]
\end{lemma}
\begin{proof}
As in the proof of Proposition \ref{prop:3.2}, we may assume $\ua=(\underbrace{0, \ldots, 0}_s, \underbrace{1, \ldots, 1}_t)$ by replacing $B$ by $B[\mathrm{diag}(\vpi^{-[a_1/2]}, \vpi^{-[a_2/2]}, \ldots, \vpi^{-[a_n/2]})]$.
Write $B$ in a block form
\[
\begin{array}{cccc}
&\hskip 5pt
\overbrace{\hphantom{B_{11}}}^{s} \,
\overbrace{\hphantom{ B_{12}}}^{t}   \\
\noalign{\vskip -17pt}
\\ 
&
B=\left(
\hskip -3pt
\begin{array}{ccl}
B_{11} & B_{12}  \\
{}^t\!B_{12} & B_{22} 
\end{array} \hskip -3pt
\right) \hskip -7pt
\begin{array}{l}
 \left.\vphantom{B_{11}} \right\} \text{\footnotesize${s}$} \\
 \left.\vphantom{B_{11}} \right\} \text{\footnotesize${t}$.} 
\end{array}
\end{array}
\]
If $\GK(B_{11})\succneqq (\underbrace{0, \ldots, 0}_s)$, then  $B[U\perp\mathbf{1}_t]\in \calm(\underbrace{0, \ldots, 0}_{s-1}, \underbrace{1, \ldots, 1}_{t+1})$ for some $U\in\GL_s(\frko)$ by the proof of Lemma \ref{lem:1.4}.
Replacing $B$ and $t$ by $B[U\perp\mathbf{1}_t]$ and $t+1$, respectively, we may assume $\GK(B_{11})=(\underbrace{0, \ldots, 0}_s)$.

Moreover, if $\GK(B_{22})\succneqq (\underbrace{1, \ldots, 1}_t)$, then we can find a non-decreasing sequence $\ua'$ and $U'\in\GL_t(\frko)$ such that $|\ua'|>|\ua|$ and $B[\mathbf{1}_s\perp U']\in \calm(\ua')$ by Lemma \ref{lem:1.4}, (1).
In this case, we go back to the case $a_n>1$.

Repeating this argument, we may assume $\ua=(\underbrace{0, \ldots, 0}_s, \underbrace{1, \ldots, 1}_t)$, $\GK(B_{11})=(\underbrace{0, \ldots, 0}_s)$, and  $\GK(B_{22})=(\underbrace{1, \ldots, 1}_t)$.
In this case, $B$ is equivalent to a reduced form of GK type $\ua$ by Proposition \ref{prop:3.4}.
Thus the lemma follows from Proposition \ref{prop:3.2}.
\end{proof}

\begin{lemma}\label{lem:3.7} % Lemma 3.8
Suppose that $B\in\calhnd_n(\frko)$.
Assume that $\ua=(a_1, \ldots, a_n)\in \bfS(\{B\})$.
If $B_1\in\calhnd_m(\frko)$ is represented by $B$, then we have
\[
|\ua^{(m)}|\leq \Del(B_1).
\]
\end{lemma}
\begin{proof}
We may assume $\ua\in S(B)$.
We can find $U\in \GL_m(\frko)$ such that $B_1[U]\in \calm(\ua^{(m)})$ by Lemma \ref{lem:1.2}.
Then we have
\[
|\ua^{(m)}|\leq \Del(B_1[U])=\Del(B_1)
\]
by Lemma \ref{lem:3.6}.
\end{proof}

\begin{lemma}\label{lem:3.8} % Lemma 3.9
Let $\ua=(a_1, a_2, \ldots, a_n)\in \ZZn$ be a sequence such that  $\ua^{(k)}= (\underbrace{0, \ldots, 0}_s, \underbrace{1, \ldots, 1}_t)$ and $a_{k+1}, \ldots, a_n\geq 1$, $s+t=k$.
Suppose that $B=(b_{ij})\in\calm(\ua)$.
Assume that $B^{(k)}$ is a reduced form of GK type $\ua^{(k)}$.
Let $(L, Q, \underline{\psi})$ be a framed quadratic module associated to $B$.
Assume that $x=\sum_{i=1}^n x_i\psi_i\in L\otimes F$ satisfies the following conditions (a), (b) and (c).
\begin{itemize}
\item[(a)] $x_1, \dots, x_k\in F$ and $x_{k+1}, \dots, x_n\in \frko$.
\item[(b)] $(x, y)_Q\in\frkp$ for any $y=\sum_{i=1}^k y_i \psi_i$, $y_1, \ldots, y_k \in\frko$.
\item[(c)] $Q(x)\in\frkp$.
\end{itemize}
Then we have $x_1, \ldots, x_s\in\frkp$ and $x_{s+1}, \dots, x_k\in\frko$.
\end{lemma}
\begin{proof}
Note that the group $G_{\ua^{(k)}}$ preserves both $\sum_{i=1}^k \frko \psi_i$ and $\sum_{i=1}^s \frkp\psi_i +\sum_{i=s+1}^k \frko\psi_i$.
Write $B$ in a block form
\[
B=\begin{pmatrix}
B^{(k)} & B_{12}  \\
{}^t\!B_{12} & B_{22} 
\end{pmatrix}.
\]
Here, $2 B_{12}\in \vpi\mathrm{M}_{k, n-k}(\frko)$ and $B_{22}\in\vpi \calh_{n-k}(\frko)$ by assumption.
By Proposition \ref{prop:3.1}, we may assume $B^{(k)}$ is of the form
\[
K_1\perp K_2\perp \cdots  \perp K_{[s/2]} \perp B' \perp \vpi (K'_1\perp  K'_2\perp \cdots \perp K_{[t/2]} ),
\]
where $K_1, K_2, \ldots K_{[s/2]}$ and $K'_1, K'_2, \ldots K'_{[t/2]}$ are primitive unramified binary forms and $B'$ is a reduced form of GK type $\emptyset$, $(0)$, $(1)$, or $(0,1)$.
We consider only the case $\GK(B')=(0,1)$, since the other cases are similar.
In this case, the condition (b) is equivalent to
\[
\ord(\sum_{i=1}^n 2b_{ij} x_i)\geq 1 \quad \text{ for } j=1, 2, \ldots, k.
\]
It follows that 
\[
x_1, x_2, \ldots, x_{s-1}\in \frkp, \quad 
x_{s+2}, x_{s+3}, \ldots, x_k \in \frko.
\]
We fix $x_1, x_2, \ldots, x_{s-1}\in \frkp$ and $x_{s+2}, x_{s+3}, \ldots, x_n \in \frko$.
We need to show $x_s\in\frkp$ and $x_{s+1}\in\frko$.
Put $E=F(\sqrt{D_{B'}})$.
Then $E$ is a ramified quadratic extension of $F$.
Moreover, there exists a prime element $\vpi_E$ of $E$ such that the framed quadratic module associated to $B'$ is weakly isomorphic to $(\frko_E, \mathrm{N}, (1, \vpi))$.
By multiplying $B$ by some unit, we may assume $B'\left[ \begin{pmatrix} x_s \\ x_{s+1}\end{pmatrix}\right]=\mathrm{N}(x_s+\vpi_E x_{s+1})$.
Put $X=x_s+\vpi_E x_{s+1}\in E$.
Then the condition (c) implies
\[
\mathrm{N}(X)+\beta_1 x_s+\bet_2 x_{s+1}
\in \frkp,
\]
where, $\bet_1=\sum_{i=k+1}^n b_{si}x_i\in\frkp$ and  $\bet_2=\sum_{i=k+1}^n b_{s+1\, i} x_i\in\frkp$.
Note that 
\[
\ord(\mathrm{N}(X))=2\ord_E(X), \qquad \ord(\beta_1 x_s+\bet_2 x_{s+1})\geq \left[\frac{\ord_E(X)}2\right]+1,
\] 
where $\ord_E$ is the order for $E$.
It follows that $X\in\frkp_E$, and so $x_s\in\frkp$ and $x_{s+1}\in\frko$.
\end{proof}

\begin{lemma} % Lemma 3.10
\label{lem:3.9}
Let $\ua=(a_1, a_2, \ldots, a_n)\in \ZZn$ be a sequence.
Put $A=\max(a_1, a_2, \ldots, a_k)$.
We assume that $a_{k+1}, \ldots, a_n\geq A$.
Suppose that $B=(b_{ij})\in\calm(\ua)$ and that $B^{(k)}$ is a reduced form of GK type $\ua^{(k)}$.
Let $(L, Q, \underline{\psi})$ be the framed quadratic module associated to $B$.
Assume that $x=\sum_{i=1}^n x_i\psi_i\in L\otimes F$ satisfies the following conditions  (a), (b) and (c).
\begin{itemize}
\item[(a)] $x_1, \dots, x_k\in F$ and $x_{k+1}, \dots, x_n\in \frko$.
\item[(b)] $\ord((\psi_j, x)_Q)\geq (a_j+A)/2$ for $j=1, \ldots, k$. 
\item[(c)] $\ord(Q(x))\geq A$.
\end{itemize}
Then we have
\[
\ord(x_i)\geq \frac{A-a_i}2 \qquad (i=1,2,\ldots,k).
\]
\end{lemma}
\begin{proof}
By multiplying $B$ by $\vpi$ if necessary, we may assume $A$ is odd.
The condition (b) is equivalent to 
\[
\ord(\sum_{i=1}^k 2b_{ij} x_i)\geq \frac{a_j+A}2\quad \text{ for } j=1,2, \ldots, k.
\]
Put
\[
B'=B[\diag(\vpi^{-[a_1/2]}, \vpi^{-[a_2/2]}, \ldots, \vpi^{-[a_k/2]}, \vpi^{-[A/2]}, \ldots, \vpi^{-[A/2]})]
\]
and
\[
x'={}^t (x'_1, \ldots, x'_n), \qquad x'_i=\begin{cases} \vpi^{[a_i/2]-[A/2]} x_i & \text{ if $i\leq k$,} \\
x_i & \text{ if $k<i\leq n$.}\end{cases}
\]
Then the conditions (b) and (c) are equivalent to the following conditions (b') and (c'), respectively.
\begin{itemize}
\item[(b')] $\ord(\sum_{i=1}^k 2b'_{ij} x'_i)\geq (a_j+1)/2-[a_j/2]$.
\item[(c')] $\ord(B'[x'])\geq 1$.
\end{itemize}
Changing the coordinate, we may assume $\ua^{(k)}$ is of the form $(\underbrace{0, \ldots, 0}_s, \underbrace{1, \ldots, 1}_t)$.
In this case, the lemma follows from Lemma \ref{lem:3.8}.
\end{proof}

Suppose that $\ua=(a_1, a_2, \ldots, a_n)\in\ZZn$ is a non-decreasing sequence.
Let $B\in\calh_n(\frko)$ be a reduced form of GK type $\ua$ and $(L, Q, \underline{\psi})$ the framed quadratic module associated to $B$.
We define $L_s$ and $\call_s$ as in section \ref{sec:1}, i.e., 
\begin{align*}
L_s&=\sum_{i=n^\ast_{s-1}+1}^n \frko\psi_i, \\
\call_s&=L_s+\sum_{u=1}^{s-1} \vpi^{\lceil (a^\ast_s-a^\ast_u)/2\rceil} L_u.
\end{align*}
\begin{lemma} % Lemma 3.11
\label{lem:3.10}
Let $B$ and $(L, Q, \underline{\psi})$ be as above.
Suppose that $x\in L$.
Then $x\in \call_s$ if and only if the following conditions (1) and (2) are satisfied.
\begin{itemize}
\item[(1)] For any $y\in \call_t$, we have 
\[
\ord((x, y)_Q)\geq {\displaystyle \frac{a^\ast_t+a^\ast_s}2} \quad \text{ for } t=1,2, \ldots, s-1.
\]
\item[(2)] $\ord(Q(x))\geq a^\ast_s$.
\end{itemize}
\end{lemma}
\begin{proof}
We denote by $M$ the set of all $x\in L$ which satisfies (1) and (2).
By Lemma \ref{lem:1.3}, we have $\call_s\subset M$.
Conversely, $M\subset \call_s$ by Lemma \ref{lem:3.9}, since $L_t\subset \call_t$.
\end{proof}

The following lemma will be used in our forthcoming paper \cite{ikedakatsurada}.
Recall that $e=\ord(2)$.
\begin{lemma}
\label{lem:3.12}
Let $\ua=(a_1, a_2, \ldots, a_n)\in \ZZn$ be a sequence and $\sig\in\frkS_n$ an $\ua$-admissible involution.
Let $B=(b_{ij})\in \calh_n(\frko)$ be a reduced form of GK type $(\ua, \sig)$.
We assume $n$ is even and $a_1 + \cdots + a_n$ is odd. 
Put $B^{-1} = (b'_{ij})$ and $\ord(\frkD_B)=d$.
Then we have the following.
\begin{align*}
\mathrm{(a)}\,\, &\ord(b'_{ii})= 2e+1-d-a_i & (i\in\calp^0(\sig)). \\
\mathrm{(b)}\,\,  &
\ord(b'_{ij})\geq (2e+1-d-a_i-a_j)/2 & (i, j\in\calp^0(\sig)). \\
\mathrm{(c)}\,\,  &
\ord(b'_{ij}) >  (2e+1-d-a_i-a_j)/2\qquad & (i\in\calp^0(\sig),\, j\notin\calp^0(\sig)). \\
\mathrm{(d)}\,\,  &
\ord(b'_{ii}) > 2e+1-d-a_i & (i\notin\calp^0(\sig)). \\
\mathrm{(e)}\,\,  &
\ord(b'_{ij}) > (2e+1-d-a_i-a_j)/2 & (i,j\notin\calp^0(\sig)).
\end{align*}
\end{lemma}
\begin{proof}
Note that $1<d \leq 2e+1$ by our assumption.
We first consider the case $n=2$.
In this case, $\ord(b_{ii})= a_i$ for $i=1,2$.
By Corollary \ref{cor:2.1}, we have 
\[
\ord(b_{12}^2-b_{11}b_{22})=a_1+a_2+d-2e-1.
\]
Since $d\leq 2e+1$, 
we have $\ord(b_{11}b_{22})=a_1+a_2\geq \ord(b_{12}^2-b_{11}b_{22})$.
Hence we have $\ord(b_{12})\geq (a_1+a_2+d-2e-1)/2$.
This proves the lemma for $n=2$.
Note that $\ord(b_{ii})\geq (2a_i+d-2e-1)/2$ for $i=1,2$ also holds, since $d\leq 2e+1$.

Now we consider the case $n>2$.
Without loss of generality, we may assume $\calp^0(\sig)=\{n-1, n\}$.
Write $B$ in a block form
\[
\begin{array}{cccc}
&\hskip -10pt
\overbrace{\hphantom{B_{11}}}^{n-2} \,
\overbrace{\hphantom{ B_{12}}}^{2}   \\
\noalign{\vskip -17pt}
\\ 
&
B=\left(
\hskip -3pt
\begin{array}{ccl}
B_{11} & B_{12}  \\
{}^t\!B_{12} & B_{22} 
\end{array} \hskip -3pt
\right) \hskip -7pt
\begin{array}{l}
 \left.\vphantom{B_{11}} \right\} \text{\footnotesize${n-2}$} \\
 \left.\vphantom{B_{11}} \right\} \text{\footnotesize${2}$.} 
\end{array}
\end{array}
\]
Then $B_{11}$ is a reduced form of GK type $(\ua', \sig')$, where $a'_i=a_i$ and $\sig'(i)=\sig(i)$ for $i=1, \ldots, n-2$.
In particular, $\calp^0(\sig')=\emptyset$ and $\ord(\frkD_{B_{11}})=0$.
Put $X=-B_{11}^{-1}B_{12}$ and $U=(u_{ij})=\begin{pmatrix} \mathbf{1} & X \\ O & \mathbf{1} \end{pmatrix}$.
Then we have
\[
B[U]=
\begin{pmatrix} B_{11} & O \\ O & T \end{pmatrix},
\]
where $T=B_{22}+B_{11}[X]$.

By Lemma \ref{lem:3.2a}, we have $\ord(u_{ij})> (a_j-a_i)/2$ for $i\in\{1,\ldots, n-2\}$ and $j\in\{n-1, n\}$.
Hence we have $B_{11}[X]\in \calm^0((a_{n-1}, a_n))$.
It follows that $T$ is a reduced form of GK type $(a_{n-1}, a_n)$.
Note that $\ord(\frkD_T)=\ord(\frkD_B)=d$, since $\ord(\frkD_{B_{11}})=0$.
Hence we have
\begin{align*}
\ord(t'_{ii})\geq&  2e+1-d-a_{i+n-2} \qquad &(i=1,2), \\
\ord(t'_{12})\geq&  (2e+1-d-a_{n-1}-a_n)/2,&
\end{align*}
where $(t'_{ij})=T^{-1}$.
This proves (a) and (b), since $t'_{ij}=b'_{i+n-2, j+n-2}$.
As we have observed as above, these two inequalities imply
\[
\ord(t'_{ij})\geq  (2e+1-d-a_{i+n-2}-a_{j+n-2})/2 \qquad (i, j=1,2).
\]
Next, we prove (c).
Since
\[
B^{-1}=
\begin{pmatrix} \mathbf{1} & -X \\ O & \mathbf{1} \end{pmatrix}
\begin{pmatrix} B_{11}^{-1} & O \\ O & T^{-1} \end{pmatrix}
\begin{pmatrix} \mathbf{1} & O \\ -{}^t\! X & \mathbf{1} \end{pmatrix},
\]
we have
\begin{align*}
\ord(b'_{ij})\geq & \min\{ \ord(u_{i, n-1}t'_{1, j-n+2}), \ord(u_{i, n}t'_{2, j-n+2})\} \\
> &
(2e+1-d-a_i-a_j)/2
\end{align*}
for $i\in\{1, \ldots, n-2\}$ and $j\in\{n-1, n\}$.
Hence we have (c).

By Lemma \ref{lem:3.2a}, we have 
\begin{align*}
\ord((B_{11})_{ii}) \geq& 2e-a_i, \qquad &(i=1,\ldots, n-2), \\
\ord((B_{11})_{ij}) \geq& (2e-a_i-a_j)/2, \qquad &(i,j=1,\ldots, n-2).
\end{align*}
Here, $(B_{11})_{ij}$ is the $ij$-th entry of $B_{11}$. 
One can easily show
\begin{align*}
\ord((XT\,{}^t\!X)_{ii})\geq & 2e+1-d-a_i,  \\
\ord((XT\,{}^t\!X)_{ij})\geq & (2e+1-d-a_i-a_j)/2, 
\end{align*}
for $i,j=1,\ldots, n-2$.
Hence we have (d) and (e).
\end{proof}

\section{Reduction theorem} % Section 4
\label{sec:4}

Suppose that $\ua=(a_1, \ldots, a_n)\in\ZZ_{\geq 0}^n$ is a non-decreasing sequence.
In this case, the integers $n_1, n_2, \dots, n_r$ are given by
\begin{align*}
&a_1=\cdots=a_{n_1} <a_{n_1+1}, \\
&a_{n_1}<a_{n_1+1}=\cdots=a_{n_1+n_2}<a_{n_1+n_2+1}, \\
&\cdots \\
&a_{n_1+\cdots+n_{r-1}}<a_{n_1+\cdots+n_{r-1}+1}=\cdots =a_{n_1+\cdots+n_r}
\end{align*}
with $n=n_1+n_2+\cdots+n_r$.
For $s=1, 2, \ldots, r$, we set
\[ 
n^\ast_s=\sum_{v=1}^{s} n_v, \qquad a^\ast_s=a_{n^\ast_{s-1}+1}=\cdots=a_{n^\ast_s}
\]
We set $n^\ast_0=0$.
The $s$-th block $I_s$ is defined by $I_s=\{n^\ast_{s-1}+1, n^\ast_{s-1}+2, \ldots, n^\ast_s\}$.

We say that two $\ua$-admissible involutions $\sig, \sig'\in\frkS_n$ are equivalent if they are conjugate by an element of $\frkS_{n_1}\times \cdots \times\frkS_{n_r}$.
In each equivalence class of $\ua$-admissible involutions, there exists a unique $\ua$-admissible involution $\sig$ satisfying the following properties (i), (ii) and (iii).
\begin{itemize}
\item[(i)] 
If $i \in  \calp_s^0\cup\calp_s^-$, then $i$ is the maximal element of $I_s$.
\item[(ii)] 
If $i \in  \calp_s^+$, then $i$ is the minimal element of $I_s$.
\item[(iii)]
If $a_i=a_{\sig(i)}$, then $|\sig(i)-i|\leq 1$.
\end{itemize}
We say that an $\ua$-admissible involution $\sig$ is standard if $\sig$ satisfies these conditions.
Thus the set of standard $\ua$-admissible involutions is a complete set of representatives for the equivalence classes of $\ua$-admissible involutions.
We shall say that a GK type $(\ua, \sig)$ is a standard GK type if $\sig$ is standard.

Thus an involution $\sig\in\frkS_n$ is a standard $\ua$-admissible involution if the following conditions are satisfied:
\begin{itemize}
\item[(i)] 
$\calp^0$ has at most two elements.
If $\calp^0$ has two distinct elements $i$ and $j$, then $a_i\not\equiv a_j \text{ mod $2$}$.
Moreover, if $i \in  I_s\cap \calp^0$, then $i$ is the maximal element of $I_s$, and
\[
i=\max\{ j \, |\, j\in \calp^0\cup\calp^+, \, a_j\equiv a_i \text{ mod }2\}.
\] 
\item[(ii)]
For $s=1, \ldots, r$,  there is at most one element in $I_s\cap\calp^-$.
If $i \in  I_s\cap\calp^-$, then $i$ is the maximal element of $I_s$ and 
\[
\sig(i)=\min\{j\in \calp^+ \,| \, j>i,\, a_j\equiv a_i \text{ mod } 2\}.
\]
\item[(iii)]
For $s=1, \ldots, r$,  there is at most one element in $I_s\cap\calp^+$.
If $i \in  I_s\cap\calp^+$, then $i$ is the minimal element of $I_s$ and
\[
\sig(i)=\max\{j\in \calp^- \,| \, j<i,\, a_j\equiv a_i \text{ mod } 2\}.
\]
\item[(iv)]
If $a_i=a_{\sig(i)}$, then $|i -\sig(i)| \le 1$. 
\end{itemize}

We draw a picture of an example of a standard GK type.
Let us consider a standard GK type given by
\begin{align*}
\ua&=(0,0,0, 1, 2,2,2,2, 3,3, 5,5, 6,6,6,6, 7,7,7), \\
\sig&=(12) (35) (4, 17) (67) (8, 13) (9,10) (11, 12) (14, 15) (18, 19).
\end{align*}
Then this GK type can be picturized as follows.

\bigskip
{\small
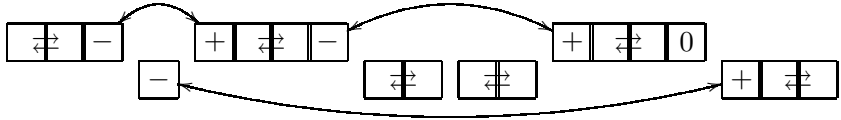
\begin{xy}
(0,5) *+[F]{\phantom{+}};
(5,5) *+[F]{\phantom{+}} ;
(2.5,4.5) *+[]{\rightleftarrows};
(10,5) *+[F]{-}="A";
(25,5) *+[F]{+}="B";
(30,5) *+[F]{\phantom{+}} ;
(32.5,4.5) *+[]{\rightleftarrows};
(35,5) *+[F]{\phantom{+}} ; ;
(40,5) *+[F]{-}="C" ;
(72.5,5) *+[F]{+}="D";
(77.5,5) *+[F]{\phantom{+}} ;
(80,4.5) *+[]{\rightleftarrows};
(82.5,5) *+[F]{\phantom{+}} ;
(87.5,5) *+[F]{\phantom{+}};
(87.7,5) *+[]{0};
(17.5,0) *+[F]{-} ="E";
(47.5,0) *+[F]{\phantom{+}};
(50,-0.5) *+[]{\rightleftarrows};
(52.5,0) *+[F]{\phantom{+}} ;
(60,0) *+[F]{\phantom{+}};
(62.5,-0.5) *+[]{\rightleftarrows};
(65,0) *+[F]{\phantom{+}};
(95,0) *+[F]{+} ="F";
(100,0) *+[F]{\phantom{+}};
(102.5,-0.5) *+[]{\rightleftarrows};
(105,0) *+[F]{\phantom{+}};
\ar @/_5mm/ @{<->} "C";"D";
\ar @/^5mm/ @{<->} "B";"A";
\ar @/^5mm/ @{<->} "E";"F";
\end{xy}
}
\bigskip\medskip
\noindent
Here, the upper line shows blocks $I_s$ with $a^\ast_s$ even, and the lower line shows blocks $I_s$ with $a^\ast_s$ odd.

\bigskip

Let  $(\ua, \sig)$ be a standard GK type.
For $1\leq k\leq n$, we define $\sig^{(k)}\in\frkS_k$ by
\[
\sig^{(k)}(i)=
\begin{cases}
i & \text{ if $\sig(i)>k$,} \\
\sig(i) & \text{ otherwise.}
\end{cases}
\]
If $\sig^{(k)}$ is $\ua^{(k)}$-admissible, then $\sig^{(k)}$ is also standard.
In this case, we say that the standard GK type $(\ua^{(k)}, \sig^{(k)})$ is a restriction of the standard GK type $(\ua, \sig)$.
We also say $(\ua, \sig)$ is an extension of $(\ua^{(k)}, \sig^{(k)})$.
The proof of the following lemma is easy and omitted.
\begin{lemma} % Lemma 4.1
\label{lem:4.1}
Let $(\ua, \sig)$ be a standard GK type, and $B=(b_{ij})\in\calh_n(\frko)$ a reduced form of standard GK type $(\ua, \sig)$.
\begin{itemize}
\item[(1)]
If $a_k<a_{k+1}$, then $\sig^{(k)}$ is $\ua^{(k)}$-admissible and  $B^{(k)}$ is a reduced form of GK type $(\ua^{(k)}, \sig^{(k)})$.
\item[(2)]
If $n\in \calp^0\cup\calp^+$, then $\sig^{(n-1)}$ is $\ua^{(n-1)}$-admissible and $B^{(n-1)}$ is a reduced form of GK type $(\ua^{(n-1)}, \sig^{(n-1)})$.
\item[(3)]
If $a_{n-1}=a_n$ and if $\sig(n)=\sig(n-1)$, then $\sig^{(n-2)}$ is $\ua^{(n-2)}$-admissible and $B^{(n-2)}$ is a reduced form of GK type $(\ua^{(n-2)}, \sig^{(n-2)})$.
\end{itemize}
\end{lemma}

\begin{lemma} %Lemma 4.2 
\label{lem:4.2}
Suppose that $B=(b_{ij})\in\calm(\ua)$, where  $\ua=(a_1, \dots, a_n)$.
Write $B$ in a block form
\[
B=\begin{pmatrix}
B^{(m)} & C \\
{}^t C & D
\end{pmatrix} =(b_{ij}).
\]
We assume that $B^{(m)}$ is a reduced form of GK type $(\ua^{(m)}, \sig_m)$ for some $\ua^{(m)}$-admissible involution $\sig_m\in \frkS_m$.
Then there exists $U\in G_\ua^\bigtriangleup$ 
satisfying the following conditions.
\begin{itemize}
\item[(1)] $U$ is of the form
$U=\begin{pmatrix} \mathbf{1}_m & X \\ 0 & \mathbf{1}_{n-m} \end{pmatrix}$. 
\item[(2)]
Put
\[
B[U]
=\begin{pmatrix} B^{(m)} & C' \\ {}^t C' & D' \end{pmatrix} .
\]
Then $i$-th row of $C'$ is $0$ unless $i\in \calp^0(\sig_m)$.
\end{itemize}
\end{lemma}
\begin{proof}
We first consider the case when $\calp^0(\sig_m)=\emptyset$.
Put 
\[
X=-(B^{(m)})^{-1}C, \quad U=\begin{pmatrix} \mathbf{1}_m & X \\ 0 & \mathbf{1}_{n-m} \end{pmatrix}.
\]
Put $(B^{(m)})^{-1}=(y_{ij})$.
By Lemma  \ref{lem:3.2a}, $\mathrm{ord}(2^{-1}y_{ij})\geq -(a_i+a_j)/2$ for $1\leq i \leq m$ and $1\leq j\leq m$.
On the other hand, $\mathrm{ord}(2c_{ij})\geq (a_i + a_{m+j})/2$ for $1\leq i\leq m$ and $1\leq j\leq n-m$, where $c_{ij}$ is the $ij$-th entry of $C$.
Hence we have
\[
\mathrm{ord}(x_{ij})\geq
\min_{1\leq k\leq m} \{ \mathrm{ord}(y_{ik} c_{kj}) \}
\geq (a_{m+j}-a_i)/2,
\]
where $x_{ij}$ is is the $ij$-th entry of $X$.
It follows that $U\in G_\ua^\bigtriangleup$.
Thus in this case, the conditions are satisfied.

Next, we consider the case $\calp^0(\sig_m)\neq \emptyset$.
For simplicity, we assume $\sharp\calp^0(\sig)=1$.
Put $\calp^0(\sig_m)=\{i_0\}$.
Write
\[
\begin{array}{cccc}
& \hskip 12pt
\text{\footnotesize$\begin{matrix} i_0 \\ \noalign{\vskip-5pt} \vdots \end{matrix}$}   \\ & 
B^{(m)}=\left(
\begin{array}{ccc l}
B_{11} & \ast & B_{12} &  \\
\ast & b_{i_0 \, i_0} & \ast \\
{}^t\!B_{12} & \ast & B_{22}  
\end{array} \hskip -10pt
\right) \hskip -8pt
\begin{array}{l}
 \vphantom{B_{11} } \\
 \vphantom{B_{11}} \text{\footnotesize$\cdots {i_0}$,} \\
 \vphantom{B_{11} }
\end{array}
\end{array}
\begin{array}{cccc}
\phantom{\text{\footnotesize$\begin{matrix} i_0 \\ \noalign{\vskip -2pt} \vdots \end{matrix}$}  }
\\
C=\left( \hskip -5pt
\begin{array}{c l}
C_1 \\
\ast \\
C_2
\end{array} \hskip -5pt
\right) \hskip -8pt
\begin{array}{l}
 \vphantom{B_{11} } \\
 \vphantom{B_{11}} \text{\footnotesize$\cdots {i_0}$.} \\
 \vphantom{B_{11} }
\end{array}
\end{array}
\]
Then $\begin{pmatrix} B_{11} & B_{12} \\ {}^t\! B_{12} & B_{22}\end{pmatrix}$ is a reduced form with GK type $(\ua', \sig')$, where $(\ua', \sig')$ is the GK type obtained by removing $i_0$-th component from $(\ua^{(m)}, \sig_m)$.
In particular, $\calp^0(\sig')=\emptyset$.
Put 
\[
\begin{pmatrix} X_1 \\ X_2 \end{pmatrix}
=
-\begin{pmatrix} B_{11} & B_{12} \\ {}^t\! B_{12} & B_{22}\end{pmatrix}^{-1}  
\begin{pmatrix} C_1 \\ C_2 \end{pmatrix}
\]
and
\[
\begin{array}{cccc}
X=\left( \hskip -5pt
\begin{array}{c l}
X_1 \\
0 \\
X_2
\end{array} \hskip -5pt
\right) \hskip -8pt
\begin{array}{l}
 \vphantom{B_{11} } \\
 \vphantom{B_{11}} \text{\footnotesize$\cdots {i_0}$.} \\
 \vphantom{B_{11} }
\end{array}
\end{array}
\]
Then $U=\begin{pmatrix} \mathbf{1}_m & X \\ 0 & \mathbf{1}_{n-m} \end{pmatrix}$ satisfies the required conditions.
The case $\sharp \calp^0(\sig_m)=2$ can be treated in a similar way.
\end{proof}

The following lemma will be used in our forthcoming paper
\begin{lemma} %Lemma 4.3
\label{lem4.3}
Let the notation be as in Lemma \ref{lem:4.2}.
Put $B[U]=(b_{ij}')$.
If $\mathrm{ord}(b_{m+1,m+1}) >a_{m+1}$ and $\mathrm{ord}(2b_{i,m+1}) > (a_i+a_{m+1})/2$ for any $1 \le i \le m$, then we have $\mathrm{ord}(b_{m+1,m+1}') > a_{m+1}$.
\end{lemma}
\begin{proof}
If $\mathrm{ord}(2b_{i,m+1}) > (a_i+a_{m+1})/2$ for any $1 \le i \le m$, then $\mathrm{ord}(x_{i1})>(a_{m+1}-a_i)/2$ for any $1 \leq i \leq m$, and so $\mathrm{ord}(b_{m+1,m+1}') > a_{m+1}$.
\end{proof}

\begin{lemma} % Lemma 4.4
\label{lem:4.4}
Suppose that $a_1, a_2\in\ZZ_{\geq 0}$ and $a_2-a_1$ is an even integer.
We assume 
\[
\ord(b_{11})=a_1, \quad \ord(2b_{12})>\frac{a_1+a_2}2, \quad \ord(b_{22})=a_2.
\]
Then there exists $x\in F$ such that
\[
\ord(x)\geq \frac{a_2 - a_1}2, \quad \ord(b_{22}+2b_{12}x+b_{11}x^2)>a_2.
\]
\end{lemma}
\begin{proof}
It is enough to consider the case $a_1=a_2=0$.
We denote the image of $b_{11}$ and $b_{22}$ in $\frkk=\frko/\frkp$ by $\bar b_{11}$ and $\bar b_{22}$, respectively.
Since $\frkk$ is a finite field of characteristic $2$, there exists $t\in \frkk$ such that $\bar b_{22}=\bar b_{11} t^2$.
Then one can choose $x\in\frko$ such that $\bar x=t$.
\end{proof}

\begin{lemma} %Lemma 4.5
\label{lem:4.5}
Let $\ua=(a_1, a_2, \ldots, a_n)\in\ZZn$ be a non-decreasing sequence.
Suppose that $B=(b_{ij})\in\calhnd_n(\frko)$ is optimal and  $\GK(B)=\ua$.
We assume that $m < n$ and $B^{(m)}$ is a reduced form of a standard GK type $(\ua^{(m)}, \sig_m)$ for some $\ua^{(m)}$-admissible involution $\sig_m\in\frkS_m$.
Then there exists a standard GK type $(\ua^{(k)}, \sig_k)$ and $U\in G_\ua^\bigtriangleup$ satisfying the following conditions.
\begin{itemize}
\item[(1)]
$k>m$ and $(\ua^{(k)}, \sig_k)$ is an extension of $(\ua^{(m)}, \sig_m)$.
\item[(2)]
$B[U]^{(k)}$ is a reduced form of GK type $(\ua^{(k)}, \sig_k)$.
\end{itemize}
\end{lemma}
\begin{proof}
Put $c=a_{m+1}$.
Let $s$ be the maximal integer such that $c=a_{m+1}=\cdots=a_{m+s}$.
Write $B$ in a block form as follows.
\[
\begin{array}{ccccc}
&\hskip -30pt
\overbrace{\hphantom{B_{11}}}^{m} \;\;
\overbrace{\hphantom{ B_{12}}}^{s}  \;
\overbrace{\hphantom{ B_{15}}}^{n-m-s} \\ & 
B=\left(
\begin{array}{cccl}
B_{11} & B_{12} & B_{13}  \\
{}^t\! B_{12} & B_{22} & B_{23}  \\
{}^t\! B_{13} & {}^t\! B_{23} & B_{33}  \\
\end{array} \hskip -0pt
\right) \hskip -5pt
\begin{array}{l}
 \left.\vphantom{B_{11}} \right\} \text{\footnotesize${m}$} \\
 \left.\vphantom{B_{11}} \right\} \text{\footnotesize${s}$} \\
 \left.\vphantom{B_{11}} \right\} \text{\footnotesize${n-m-s}$.} 
\end{array}
\end{array}
\]
By Lemma \ref{lem:4.2}, we may assume
\[
(B_{12})_{ij}=0
\quad \text{ for } 1\leq i\leq m,\; i\notin \calp^0(\sig_m).
\]
Suppose that there exists $h\in\calp^0(\sig_m)$ such that
\[
\min_{1\leq j\leq s} (\ord(2(B_{12})_{hj})) = \frac{a_h+c}2.
\]
We claim that
\[
\min_{1\leq j\leq s}(\ord(2(B_{12})_{ij})) > \frac{a_i+c}2 \quad\text{ for } i\neq h.
\]
In fact, if $h'\in \calp^0(\sig_m)$ and $h'\neq h$,  then $a_h\not\equiv a_{h'}$ mod $2$.
It follows that $(a_{h'}+c)/2\notin \ZZ$, and so $\ord(2(B_{12})_{hj})>(a_{h'}+c)/2$ for $j=1,2,\dots, s$.

By changing the coordinates, we may assume 
\[
\ord(2(B_{12})_{h1})=\ord(2b_{h, m+1})=\frac{a_h+c}2.
\]
In this case, put $k=m+1$ and
\[
\sig_k(i)=
\begin{cases}
i & 1\leq i\leq m, \; i\neq h, \\
m+1, & i=h \\
h & i=m+1 
\end{cases}
\]
Then $(\ua^{(k)}, \sig_k)$ is a standard GK type, which is an extension of $(\ua^{(m)}, \sig_m)$.
Moreover, $B^{(k)}$ is a reduced form of GK type $(\ua^{(k)}, \sig_k)$.

Next, we consider the case
\[
\min_{1\leq j\leq s}(\ord(2(B_{12})_{ij})) > \frac{a_i+c}2 \quad\text{ for any } 1\leq i\leq m.
\]
In this case, we have
\[
\begin{pmatrix}
0 & B_{12} & 0  \\
{}^t\! B_{12} & 0 & 0  \\
0 & 0 & 0
\end{pmatrix}
\in\calm^0(\ua)
\]
and $\GK(B_{22})=(\underbrace{c, \dots, c}_s)$ by Lemma \ref{lem:1.4}.
If $s\geq 2$, then we may assume $B_{22}=\vpi^c K\perp B'$, for some primitive unramified binary form $K$ and $B'\in\calh_{s-2}(\frko)$ by Proposition \ref{prop:3.4}.
In this case, put $k=m+2$ and
\[
\sig_k(i)=
\begin{cases}
i & 1\leq i\leq m,  \\
m+2 & i=m+1 \\
m+1 & i=m+2.
\end{cases}
\]
Then $(\ua^{(k)}, \sig_k)$ is a standard GK type, which is an extension of $(\ua^{(m)}, \sig_m)$.
Moreover, $B^{(k)}$ is a reduced form of GK type $(\ua^{(k)}, \sig_k)$.

Finally, we consider the case when
\[
\begin{pmatrix}
0 & B_{12} & 0  \\
{}^t\! B_{12} & 0 & 0  \\
0 & 0 & 0
\end{pmatrix}
\in\calm^0(\ua)
\]
and $s=1$.
Note that $B_{22}=(b_{m+1, m+1})$ and $\ord(b_{m+1, m+1})=c$.
In this case, we claim that $\{ h\in \calp^0(\sig_k)\, |\, a_h\equiv c \text{ mod }2\}=\emptyset$.
Suppose $h\in \calp^0(\sig_k)$ and $a_h\equiv c$ mod $2$.
Then there exists $x\in \frko$ such that 
\[
\ord(x)\geq \frac{c-a_h}2, \quad \ord(b_{m+1, m+1}+2b_{h, m+1}x+b_{hh}x^2)>c
\]
by Lemma \ref{lem:4.4}.
Put $B'=B[U]$, where $U$ is the upper triangular unipotent matrix whose $U_{h, m+1}=x$ and $U_{ij}=0$ for $i<j$, $(i, j)\neq (h, m+1)$. 
Then $U\in G^\bigtriangleup_\ua$ and $\GK(B')\succneqq \ua$ by Lemma \ref{lem:1.4}.
This contradicts the assumption $\GK(B)=\ua$.
Thus we have $\{ h\in \calp^0(\sig_k)\, |\, a_h\equiv c \text{ mod }2\}=\emptyset$.
In this case, put $k=m+1$ and
\[
\sig_k(i)=
\begin{cases}
i & 1\leq i\leq m,  \\
m+1 & i=m+1.
\end{cases}
\]
Then $(\ua^{(k)}, \sig_k)$ is a standard GK type, which is an extension of $(\ua^{(m)}, \sig_m)$.
Moreover, $B^{(k)}$ is a reduced form of GK type $(\ua^{(k)}, \sig_k)$.
\end{proof}

By using Lemma \ref{lem:4.5} repeatedly, we obtain the following theorem.
\begin{theorem}[Reduction theorem] %Theorem 4.1
\label{thm:4.1}
Let $\ua=(a_1, a_2, \ldots, a_n)\in\ZZn$ be a non-decreasing sequence.
Suppose that $B\in\calhnd_n(\frko)$ is optimal and  $\GK(B)=\ua$.
Then there exists $U\in G_\ua^\bigtriangleup$ and a standard $\ua$-admissible involution $\sig$ such that $B[U]$ is a reduced form of GK type $(\ua, \sig)$.
In particular, a non-degenerate half-integral symmetric matrix is equivalent to a reduced form.
\end{theorem}

\begin{remark} % Remark 4.1
We shall say a reduced form $B=(b_{ij})$ of GK type $(\ua, \sig)$ is a strongly reduced form if the following condition hold:
\begin{itemize}
\item[(1)]
If $i\notin \calp^0(\sig)$, then $b_{ij}=0$ for $j>\max\{ i, \sig(i)\}$.
\end{itemize}
The proof of Theorem \ref{thm:4.1} shows that a non-degenerate half-integral symmetric matrix is equivalent to a strongly reduced form.
\end{remark}

Recall that two $\ua$-admissible involutions $\sig, \sig'\in\frkS_n$ are equivalent if they are conjugate by an element of $\frkS_{n_1}\times \cdots \times\frkS_{n_r}$.
The equivalence class of $\sig$ is determined by
\[
\sharp \calp_1^+, \dots, \sharp \calp_r^+, \sharp \calp_1^-, \dots, \sharp \calp_r^-, \sharp \calp_1^0, \dots, \sharp \calp_r^0, 
\]
since
\begin{align*}
\sig(i)&=\min\{j\in \calp^+\,|\, j>i, \, a_j\equiv a_i \text{ mod } 2\} \quad \text{ for } i\in\calp^-, \\
\sig(i)&=\max\{j\in \calp^-\,|\, j<i,\, a_j\equiv a_i \text{ mod } 2\} \quad \text{ for } i\in\calp^+. 
\end{align*}
Note that for each block $I_s$, exactly one of the following possibilities occur:
\begin{itemize}
\item[(1)] $n_s$ is even and $\sharp\calp_s^+=\sharp\calp_s^-=\sharp\calp_s^0=0$.
\item[(2)] $n_s$ is even and $\sharp\calp_s^+=\sharp\calp_s^- + \sharp\calp_s^0=1$.
\item[(3)] $n_s$ is odd and $\sharp\calp_s^+=1$, $\sharp\calp_s^-=\sharp\calp_s^0=0$.
\item[(4)] $n_s$ is odd and $\sharp\calp_s^+=0$, $\sharp\calp_s^- + \sharp\calp_s^0=1$.
\end{itemize}
Moreover, if $i\in\calp^0$, then
\[
i=\max\{ j\in \calp^0\cup \calp^+\,|\, a_i\equiv a_j \text{ mod } 2\}.
\]
It follows that the equivalence class of $\sig$ is determined by
\[
\sharp \calp_1^+, \dots, \sharp \calp_r^+. 
\]
We determine the number of equivalence classes of GK types for given $\ua$.
For a block $I_s$, let $k_s$ be the number of blocks $I_u$ such that
\[
1\leq u < s, \quad a^\ast_u\equiv a^\ast_s \text{ mod }2, \quad n_u\not\equiv 0  \text{ mod }2.
\]
If $n_s$ is odd, then the possibility (3) (resp. the possibility (4)) occurs if and only if $k_s$ is odd (resp. even).
Suppose that $n_s$ is even.
If $k_s$ is even, only the possibility (1) occurs.
If $k_s$ is odd, both (1) and (2) are possible.
Note also that $\sharp\calp_s^0=1$ if and only if $k_s$ is even and
\[
i=\max\{ j\in \calp^0\cup \calp^+\,|\, a_i\equiv a_j \text{ mod } 2\}. 
\]
Thus the number of equivalence classes of GK types is equal to $2^K$, where $K$ is the number of blocks $I_s$ such that $n_s\equiv 0$ mod $2$ and $k_s\not\equiv 0$ mod $2$.

The proof of the following lemma is easy and will be omitted.
\begin{lemma} % Lemma 4.6
\label{lem:4.6}
Suppose that $B\in\calh_n(\frko)$ is a reduced form of GK type $(\ua, \sig)$.
If $U\in N_\ua^\bigtriangledown$, then $B[U]$ is also a reduced form of GK type $(\ua, \sig)$.
\end{lemma}

\begin{theorem} %Theorem 4.2
\label{thm:4.2}
Let $B=(b_{ij})\in\calh_n(\frko)$ and $T=(t_{ij})\in\calh_n(\frko)$ be reduced forms with standard GK types $(\ua, \sig_1)$ and $(\ua, \sig_2)$, respectively.
If $B\sim T$, then we have $\sig_1=\sig_2$.
%If $\sig_1\not\sim\sig_2$ are not equivalent, then $B$ and $T$ are not equivalent.
\end{theorem}
\begin{proof}
Assume that both $B[U]=T$ for some $U\in\GL_n(\frko)$.
Since both $B$ and $T$ are optimal, we have $U\in G_\ua$.
Since $G_\ua=N_\ua^\bigtriangledown G_\ua^\bigtriangleup$, there exist $U_1\in N_\ua^\bigtriangledown$ and $U_2\in G_\ua^\bigtriangleup$ such that $U=U_1U_2$.
Note that $B[U_1]$ is a reduced form of GK type $(\ua, \sig_1)$ by Lemma \ref{lem:4.6}.
Replacing $B$ by $B[U_1]$, we may assume $U\in G_\ua^\bigtriangleup$.

Suppose that $\sig_1\not\sim\sig_2$.
Let $I_1, \dots, I_r$ be the blocks for $\ua$.
Then we have
\[
\sharp\calp_s^+(\sig_1)\neq \sharp\calp_s^+(\sig_2)
\]
for some $s$.
Let $s$ be the smallest integer with this property.
We may assume $\calp_s^+(\sig_1)\neq\emptyset$ and $\calp_s^+(\sig_2)=\emptyset$.
By replacing $B$ and $T$ by $B^{(n_s^\ast)}$ and $T^{(n_s^\ast)}$, we may assume $n=n_s$.
Put $m=n_{s-1}^\ast=n-n_s$.

Write $B$ and $T$ in block forms as follows.
\[
\begin{array}{ccc}
&\hskip 5pt
\overbrace{\hphantom{B_{11}}}^{m} \;\;
\overbrace{\hphantom{\; B_{12}}}^{n_s}  \; \; \\ & 
B=\left(
\begin{array}{ccl}
\phantom{{}^t\!} B_{11} & \phantom{{}^t\!} B_{12} &  \\
{}^t\! B_{12} & \phantom{{}^t\!} B_{22}  \\
\end{array} \hskip -10pt
\right) \hskip -5pt
\begin{array}{l}
 \left.\vphantom{B_{11}} \right\} \text{\footnotesize${m}$} \\
 \left.\vphantom{B_{11}} \right\} \text{\footnotesize${n_s,}$} \\
\end{array}
\end{array} \qquad 
\begin{array}{ccc}
&\hskip 5pt
\overbrace{\hphantom{T_{11}}}^{m} \;\;
\overbrace{\hphantom{\; T_{12}}}^{n_s}  \; \; \\ & 
T=\left(
\begin{array}{ccl}
\phantom{{}^t\!} T_{11} & \phantom{{}^t\!} T_{12} &  \\
{}^t T_{12} & \phantom{{}^t\!} T_{22}  \\
\end{array} \hskip -10pt
\right) \hskip -5pt
\begin{array}{l}
 \left.\vphantom{T_{11}} \right\} \text{\footnotesize${m}$} \\
 \left.\vphantom{T_{11}} \right\} \text{\footnotesize${n_s}$} \\
\end{array}
\end{array}
\]
Then we have $\begin{pmatrix}
0 & \phantom{{}^t}T_{12}  \\
{}^t T_{12} & 0 
\end{pmatrix}
\in \calm^0(\ua)$, since $\calp_s^+(\sig_2)=\emptyset$.
Decompose $U\in G_\ua^\bigtriangleup$ into
\[
U=
\begin{pmatrix}
\mathbf{1}_m & X \\
0 & \mathbf{1}_{n_s}
\end{pmatrix}
\begin{pmatrix}
U_{11} & 0 \\
0 & U_{22}
\end{pmatrix},
\]
where 
\[
U_{11}\in G_{\ua^{(m)}}, \quad U_{22}\in \GL_{n_s}(\frko),\quad 
\begin{pmatrix}
\mathbf{1}_m & X \\
0 & \mathbf{1}_{n_s}
\end{pmatrix}
\in N_\ua^\bigtriangleup.
\]
Put 
\[
B'=\begin{pmatrix}
B'_{11} & B'_{12}  \\
{}^t\! B'_{12} & B'_{22}
\end{pmatrix}
=
B\left[
\begin{pmatrix}
\mathbf{1}_m & X \\
0 & \mathbf{1}_{n_s}
\end{pmatrix}
\right].
\]
Then we have
\[
\begin{pmatrix}
0 & B'_{12}  \\
{}^t\! B'_{12} & 0 
\end{pmatrix}
=
\begin{pmatrix}
0 & T_{12}  \\
{}^t T_{12} & 0 
\end{pmatrix}
\left[
\begin{pmatrix}
U_{11}^{-1} & 0 \\
0 & U_{22}^{-1}
\end{pmatrix}
\right]
\in \calm^0(\ua).
\]
Since $\calp_s^+(\sig_1)\neq\emptyset$, there exists $h\in \calp^-(\sig_1)$ such that $\sig_1(h)\in I_s$.
Now look at the $h$-th row of 
\[
B'_{12}=B_{12}+B_{11}X.
\]
Put $Y=(y_{ij})=B_{11}X$.
We claim that 
\[
\ord(2y_{hj}) > \frac{a_h+a^\ast_s}2\quad\text{ for } j=1,2, \dots, n_s.
\]
In fact, 
\[
y_{hj}=\sum_{i=1}^m b_{hi} x_{ij},
\]
where $x_{ij}$ is the $(i, j)$-th entry of $X$.
Since $B$ is a reduced form of GK type $(\ua, \sig_1)$, we have 
\[
\ord(2 b_{hj}) > \frac{a_h + a_i}2, \quad \text{ for } 1\leq i\leq m, i\neq h.
\]
Note that $\ord(2 b_{hh}) > a_h$, since $F$ is dyadic.
Note also that
\[
\ord(x_{ij}) \geq \frac{a^\ast_s-a_i}2, \quad \text{ for } i=1,2, \dots, m,
\]
since $\begin{pmatrix}
\mathbf{1}_m & X \\
0 & \mathbf{1}_{n_s}
\end{pmatrix}
\in N_\ua^\bigtriangleup$.
This proves the claim.

Put $\sig_1(h)=m+k\in \calp_s^+(\sig_1)$.
Then we have
\[
\ord(2 (B_{12})_{hk})=\ord(2b_{h, m+k})=\frac{a_h+a_s^\ast}2,
\]
where $(B_{12})_{hk}$ is the $(h, k)$-th entry of $B_{12}$.
This is a contradiction.
\end{proof}

\noindent \textbf{Example}.
Suppose that $F=\QQ_2$.
Put
\[
B_1=
\begin{pmatrix}
1 & 1 & 0 \\
1 & 0 & 0 \\
0 & 0 & 4
\end{pmatrix},
\qquad
B_2=
\begin{pmatrix}
1 & 0 & 0 \\
0 & 0 & 2 \\
0 & 2 & 0
\end{pmatrix}.
\]
Then $B_1$ is a reduced form of GK type $((0,2,2), \begin{pmatrix} 1\,2\,3 \\ 2\,1\,3\end{pmatrix})$ and $B_2$ is a reduced form of GK type $((0,2,2), \begin{pmatrix} 1\,2\,3 \\ 1\,3\,2\end{pmatrix})$.
Note that $B_1^{(2)}$ is reduced, but $B_2^{(2)}$ is not.
For more examples, see Bouw \cite{Bouw} and Yang \cite{thyang}.

\section{Optimal forms}
\label{sec:5}

\begin{theorem}% Theorem 5.1
\label{thm:5.1a}
Suppose that $\ua=(a_1, a_2, \ldots, a_n)\in\ZZn$ is a non-decreasing sequence and  that $B\in\calm(\ua)$.
Then $\GK(B)=\ua$ if and only if $\Del(B^{(n^\ast_s)})=|\ua^{(n^\ast_s)}|$ for $s=1, \ldots, r$.
Moreover, $B$ is optimal in this case.
\end{theorem}

\begin{proof}
Assume that $\Del(B^{(n^\ast_s)})=|\ua^{(n^\ast_s)}|$ for $s=1, \ldots, r$.
Put $\GK(B)=\underline{b}=(b_1, \ldots, b_n)$.
Since $\ua \in S(B)$, we have $\underline{b}\succeq\ua$.
We shall show that $(a_1, \ldots, a_{n^\ast_s})=(b_1, \ldots, b_{n^\ast_s})$ for $s=0, 1, \ldots, r$ by induction.
The case $s=0$ is trivial.

Assume that $(a_1, \ldots, a_{n^\ast_s})=(b_1, \ldots, b_{n^\ast_s})$.
Then, we have $b_{n^\ast_s+1}\geq a_{n^\ast_s+1}$, since $\underline{b}\succeq\ua$.
Note that
\begin{align*}
a_{n^\ast_s+1}&=a_{n^\ast_s+2}=\cdots =a_{n^\ast_{s+1}}, \\
b_{n^\ast_s+1}&\leq b_{n^\ast_s+2}\leq \cdots \leq b_{n^\ast_{s+1}}.
\end{align*}
It follows that
\[
\sum_{i=1}^{n^\ast_{s+1}} a_i \leq \sum_{i=1}^{n^\ast_{s+1}} b_i.
\]
%Note that $B^{(n^\ast_{s+1})}$ is a reduced form.
%By Proposition \ref{prop:3.2}, we have
%\[
%\Delta(B^{(n^\ast_{s+1})}) =\sum_{i=1}^{n^\ast_{s+1}} a_i.
%\]
By applying Lemma \ref{lem:3.7} for $B_1=B^{(n^\ast_{s+1})}$, we have
\[
\sum_{i=1}^{n^\ast_{s+1}} b_i\leq 
\Delta(B^{(n^\ast_{s+1})})=
\sum_{i=1}^{n^\ast_{s+1}} a_i.
\]
It follows that $(a_1, \ldots, a_{n^\ast_{s+1}})=(b_1, \ldots, b_{n^\ast_{s+1}})$, as desired.

Conversely, assume that $\GK(B)=\ua$.
In particular, $B$ is optimal.
By Theorem \ref{thm:4.1}, there exist a reduced form $B_1$ and $U\in G_\ua^\bigtriangleup$ such that $B=B_1[U]$.
Then we have $B^{(n^\ast_s)}=B_1^{(n^\ast_s)}[U^{(n^\ast_s)}]$.
In particular, we have $\Delta(B^{(n^\ast_s)})=\Delta(B_1^{(n^\ast_s)})$.
Since $B_1^{(n^\ast_s)}$ is a reduced form of GK type $\ua^{(n^\ast_s)}$, we see $\Delta(B_1^{(n^\ast_s)})=|\ua^{(n^\ast_s)}|$.
The last part of the theorem is clear.
\end{proof}

\begin{corollary} % Corollary 5.1
\label{cor:5.1}
Suppose that $\ua=(a_1, a_2, \ldots, a_n)\in\ZZn$ is a non-decreasing sequence.
If $B\in\calh_n(\frko)$ is a reduced form of GK type $\ua$, then we have $\GK(B)=\ua$.
In particular, $B$ is optimal.
\end{corollary}

\begin{proof}
Note that $B^{(n^\ast_s)}$ is reduced with GK type $\ua^{(n^\ast_s)}$ for $s=1, \ldots, r$.
By Proposition \ref{prop:3.2}, we have $\Delta(B^{(n^\ast_s)})=|\ua^{(n^\ast_s)}|$.
By Theorem \ref{thm:5.1a}, $B$ is optimal and $\GK(B)=\ua$.
\end{proof}

\noindent \textsl{Proof of Theorem \ref{thm:0.1}.\quad}
By Theorem \ref{thm:4.1}, there exists $U\in G_\ua^\bigtriangleup$ such that $B[U]$ is a reduced form of GK type $\ua$.
By Proposition \ref{prop:3.2}, we have $|\ua|=\Del(B[U])=\Del(B)$.
\hfill $\square$\par\bigskip

\noindent\textsl{Proof of Theorem \ref{thm:0.2}.\quad}
The ``if" part is Proposition \ref{prop:1.1}.
We prove the ``only if" part.
Suppose that both $B$ and $B[U]$ are optimal with $U\in\GL_n(\frko)$.
By Theorem \ref{thm:4.1} and Corollary \ref{cor:5.1}, we may assume $B$ is a reduced form of GK type $\ua$.
By Lemma \ref{lem:3.10}, $\psi_i U\in \call_s$ for  any $i\in I_s$ ($s=1,2, \ldots, r$).
It follows that $U\in G_\ua$, as desired.
\hfill $\square$\par\bigskip

Theorem \ref{thm:0.3} follows from Theorem \ref{thm:5.1a} immediately.
\par\bigskip

\noindent \textsl{Proof of Theorem \ref{thm:0.4}.\quad}
We may assume that $B$ is a reduced form of GK type $\ua$.
By Theorem \ref{thm:0.2}, there exists an element $U\in G_\ua$ such that $B_1=B[U]$.
Since $G_\ua=N_\ua^\bigtriangledown G_\ua^\bigtriangleup$, there exist $U_1\in N_\ua^\bigtriangledown$ and $U_2\in G_\ua^\bigtriangleup$ such that $U=U_1 U_2$.
Then $B_1^{(k)}$ is equivalent to $B_1[U_2^{-1}]^{(k)}=B[U_1]^{(k)}$, since $U_2\in G_\ua^\bigtriangleup$.
Replacing $B_1$ by $B_1[U_2^{-1}]$, we may assume that $B_1=B[U_1]$ with $U_1\in N_\ua^\bigtriangledown$.
In this case, $B_1$ is a reduced form of GK type $\ua$ and $B^{(k)}-B_1^{(k)}\in \calm^0(\ua^{(k)})$.
Then the theorem follows from Proposition \ref{prop:3.3}.
\hfill$\square$

\begin{corollary} % Corollary 5.2
Suppose that $n$ is even.
Let $B\in\calh_n^{\mathrm{nd}}(\frko)$ be a half-integral symmetric matrix such that $\GK(B)=(\ua, \sig)$.
Then the following four conditions are equivalent.
\begin{itemize}
\item[(1)] 
$|\ua|$ is odd.
\item[(2)] $\sharp \calp^0(\sig)=2$.
\item[(3)] $\ord(\frkD_B)>0$.
\item[(4)] $\xi_B= 0$.
\end{itemize}
\end{corollary}
\begin{proof}
The equivalence of (1) and (2) follows from the definition of admissible involutions.
The equivalence of (1) and (4) follows from Theorem \ref{thm:0.1}.
The equivalence of (3) and (4) follows from the definition of $\xi_B$.
\end{proof}

\section{extended GK data} %Section 6
\label{sec:6}

In this section, we discuss combinatorial properties of the invariants $\GK(B)$,
 $\xi_{B^{(k)}}$ and $\eta_{B^{(k)}}$.
We do not assume $F$ is dyadic in this section.
The results of this section will be used in our forthcoming paper \cite{ikedakatsurada}.
First we introduce some definitions.
Put ${\mathcal Z}_3=\{0,1,-1 \}$. 
\begin{definition} %Definition 6.1 
\label{def:6.1}
An element $H=(a_1,\ldots,a_n;\vep_1,\ldots,\vep_n)$
of $\ZZ_{\ge 0}^n \times {\mathcal Z}_3^n$ is said to be a naive  EGK datum of length $n$  if the following conditions hold:
\begin{itemize}
\item [(N1)] $a_1 \le \cdots \le a_n$.
\item [(N2)] Assume that $i$ is even. 
Then $\vep_i \not=0$ if and only if $a_1+\cdots+a_i$ is even.
\item [(N3)] If $i$ is odd, then $\vep_i \not=0$. 
\item [(N4)] $\vep_1=1$.
\item [(N5)] If $i  \ge 3$ is odd and $a_1+\cdots + a_{i-1}$ is even, then $\vep_i=\vep_{i-2}\vep_{i-1}^{a_i+a_{i-1}}$.
\end{itemize}
We denote the set of naive EGK data of length $n$ by $\mathcal{NEGK}_n$.
\end{definition}

\begin{proposition} %Proposition 6.1
\label{prop:6.1}
Suppose that $F$ is non-dyadic field.
Let $T=(t_1)\perp \cdots \perp (t_n)$ be a diagonal matrix such that $\ord(t_1)\leq \ord(t_2)\leq \cdots \leq (t_n)$.
Put $a_i=\ord(t_i)$ and 
\[
\vep_i=\begin{cases} 
\xi_{T^{(i)}} & \text{ if $i$ is even,} \\
\eta_{T^{(i)}} & \text{ if $i$ is odd.} 
\end{cases}
\]
Then $(a_1, \ldots, a_n; \vep_1, \ldots, \vep_n)$ is a naive EGK datum.
\end{proposition}
\begin{proof}
Only (N5) needs a proof.
Suppose $i\geq 3$ is odd and $a_1+\cdots+a_{i-1}$ is even.
Then $F(\sqrt{D_{T^{(i-1)}}})/F$ is unramified, since $F$ is non-dyadic.
It follows that $\langle D_{B^{(i-1)}}, t\rangle=\xi_{B^{(i-1)}}^{\ord(t)}=\vep_{i-1}^{\ord(t)}$ for $t\in F^\times$.
By Lemma \ref{lem:3.3}, we have
\[
\eta_{T^{(i)}}=\eta_{T^{(i-2)}}\langle D_{B^{(i-1)}}, D_B D_{B^{(i-2)}}\rangle= \vep_{i-2}\vep_{i-1}^{a_i+a_{i-1}}.
\]
Hence the lemma.
\end{proof}
We set $\mathrm{NEGK}(T)=(a_1, \ldots, a_n; \vep_1, \ldots, \vep_n)$  and call it the naive EGK datum associated to $T$.

\begin{remark} % Remark 6.1
Conversely, for a given $H=(a_1, \ldots, a_n; \vep_1, \ldots, \vep_n)\in\mathcal{NEGK}_n$,  there exists a diagonal matrix 
\[
T=\diag(t_1, \ldots, t_n), \quad \ord(t_1)\leq \cdots\leq \ord(t_n),
\] 
such that $\mathrm{NEGK}(T)=H$.
The proof is easy and left to the reader.
\end{remark}

\begin{remark} % Remark 6.2
If $F$ is a dyadic field, then $(a_1, \ldots, a_n; \vep_1, \ldots, \vep_n)$ may not be a naive EGK datum for a diagonal matrix $T=\diag(t_1, \ldots, t_n)$ such that $\ord(t_1)\leq \cdots \leq \ord(t_n)$.
For example, put $T=(1)\bot (1)$.
Then $a_1=a_2=0$, $\vep_1=1$, $\vep_2=0$, and so (N2) does not hold.
\end{remark}

\begin{definition} %Definition 6.2 
\label{def:6.2}
Let $G=(n_1,\ldots, n_r; m_1, \ldots, m_r; \zeta_1, \ldots, \zeta_r)$ be an element of $\ZZ_{>0}^r \times \ZZ_{\ge 0}^r  \times {\mathcal Z}_3^r$.
Put $n^\ast_s=\sum_{i=1}^s n_i$ for $s \leq r$.
We say that $G$ is an EGK datum of length $n$  if the following conditions hold:
\begin{itemize}
\item [(E1)] $n^\ast_r=n$ and $m_1 <\cdots < m_r$.
\item[(E2)] Assume that $n^\ast_s$ is even. 
Then $\zeta_s \not=0$ if and only if $m_1 n_1+\cdots+m_s n_s$ is even.
\item [(E3)] Assume that $n^\ast_s$ is odd. 
Then $\zeta_s \not=0$. 
Moreover, we have
\begin{itemize}
\item[(a)] Assume that $n^\ast_i$ is even for any $i<s$.
Then we have
\[
\zeta_s
=\zeta_1^{m_1+m_2}
\zeta_2^{m_2+m_3}
\cdots
\zeta_{s-1}^{m_{s-1}+m_s}.
\]
In particular, $\zeta_1=1$ if $n_1$ is odd.
\item[(b)] 
Assume that $m_1n_1+\cdots + m_{s-1}n_{s-1}+m_s(n_s-1)$ is even and that $n^\ast_i$ is odd for some $i<s$.
Let $t<s$ be the largest number such that $n^\ast_t$ is odd.
Then we have
\[
\zeta_s=
\zeta_t
\zeta_{t+1}^{m_{t+1}+m_{t+2}} 
\zeta_{t+2}^{m_{t+2} + m_{t+3}}
\cdots
\zeta_{s-1}^{m_{s-1}+m_s}.
\]
In particular, $\zeta_s=\zeta_t$ if $t+1=s$.
\end{itemize}
\end{itemize}
We denote the set of EGK data of length $n$ by $\mathcal{EGK}_n$.
Thus $\mathcal{EGK}_n\subset \coprod_{r=1}^n (\ZZ_{>0}^r \times \ZZ_{\ge 0}^r  \times {\mathcal Z}_3^r)$.
\end{definition}

Let $H=(a_1,\ldots,a_n;\vep_1,\ldots,\vep_n)$ be a naive EGK datum.
Let $n_1$, $n_2$, $\dots$, $n_r$  and $n^\ast_1$, $n^\ast_2$, $\dots$, $n^\ast_r$  be as in section \ref{sec:1}.
For $s=1, 2, \ldots, r$, we set $m_s=a_{n^\ast_s}$ and $\zeta_s=\vep_{n^\ast_s}$.
The following proposition can be easily verified.
\begin{proposition} %Proposition 6.2
\label{prop:6.2}
Let $H=(a_1,\ldots,a_n;\vep_1,\ldots,\vep_n)$ be a naive EGK datum.
Then $G=(n_1, \ldots, n_r; m_1, \ldots, m_r; \zeta_1, \ldots, \zeta_r)$ is an EGK datum.
\end{proposition}
We define a map $\Ups=\Ups_n:\mathcal{NEGK}_n\rightarrow \mathcal{EGK}_n$ by $\Ups(H)=G$.
We call $G=\Ups(H)$ the EGK datum associated to a naive EGK datum $H$.
We also write $\Ups(\ua)=(n_1, \ldots, n_r; m_1, \ldots, m_r)$, if there is no fear of confusion.

\begin{proposition} %Proposition 6.3
\label{prop:6.3}
The map $\Ups:\mathcal{NEGK}_n\rightarrow\mathcal{EGK}_n$ is surjective.
Thus for any EGK datum 
\[
G=(n_1, \ldots, n_r; m_1, \ldots, m_r; \zeta_1, \ldots, \zeta_r)
\]
of length $n$, there exists a naive EGK datum $H$ such that $\Ups(H)=G$.
\end{proposition}
\begin{proof}
Note that $(a_1, \ldots, a_n)$ is determined by $(n_1, \ldots, n_r;m_1, \ldots, m_r)$.
We proceed by induction with respect to $n$.
When $n=1$, the proposition is trivial.
First consider the case $n_r=1$.
In this case, put 
\[
G'=(n_1,n_2,\ldots,n_{r-1}; m_1,\ldots, m_{r-1}; \zeta_1,\ldots,\zeta_{r-1}).
\]
Then $G'$ is an $\EGK$ datum.
By the induction hypothesis, there exists a naive EGK datum $H'=(a_1,\ldots,a_{n-1};\vep_1,\ldots,\vep_{n-1})$ such that $\Ups(H')=G'$.
Then $H= (a_1,\ldots,a_{n-1}, a_n;\vep_1,\ldots,\vep_{n-1}, \zeta_r)$ satisfies the condition.

Now, we assume $n_r \ge 2$.
We define an EGK datum 
\[
G'=(n_1,\ldots,n_r-1;m_1,\ldots,m_r;\zeta_1,\ldots,\zeta_{r-1},\zeta_r').
\]
of length $n-1$ as follows.
If $n$ and $m_1n_1+\cdots+m_r(n_r-1)$ are odd, then put $\zeta'_r=0$.

Assume that $n$ and $m_1n_1+\cdots+m_r n_r$ are even. 
If $n^\ast_i$ is even for any $i<r$, then we put
\[
\zeta_r'=
\zeta_1^{m_1+m_2}\cdots\zeta_{r-1}^{m_{r-1}+m_r}.
\]
Let $t<r$ be the largest number such that $n^\ast_t$ is odd.
Then we put
\[
\zeta_r'=
\zeta_t  \zeta_{t+1}^{m_{t+1}+m_{t+2}}\cdots  \zeta_{r-1}^{m_{r-1}+m_r}.
\]
We put $\zeta_r'=\pm 1$ arbitrarily, in other cases.
Then one can easily see $G'$ is an $\EGK$ datum.
By the induction hypothesis, there exists a naive EGK datum 
\[
H'=(a_1,\ldots,a_{n-1};\vep_1,\ldots,\vep_{n-1}).
\]
Then 
\[
H= (a_1,\ldots,a_{n-1}, a_n;\vep_1,\ldots,\vep_{n-1}, \zeta_r)
\] 
is a naive EGK datum such that $\Ups(H)=G$.
\end{proof}

\begin{definition} %Definition 6.3
\label{def:6.3}
Let  $B \in \calh_n(\frko )$ be an optimal form such that $\GK(B)=\ua$.
Put $\Ups(\ua)=(n_1, \ldots, n_r; m_1, \ldots, m_r)$.
We define $\zeta_s=\zeta_s(B)$ by 
\[
\zeta_s=\zeta_s(B)=
\begin{cases}
\xi_{B^{(n_s^\ast)}} & \text{ if $n_s^\ast$ is even,} \\
\eta_{B^{(n_s^\ast)}} & \text{ if $n_s^\ast$ is odd.} 
\end{cases}
\]
Then put
$\EGK(B)=(n_1,\ldots,n_r;m_1,\ldots,m_r;\zeta_1,\ldots,\zeta_r)$.
For $B \in \calhnd_n(\frko )$, we define $\EGK(B)=\EGK(B')$, where $B'$ is an optimal form equivalent to $B$.
\end{definition}
By Theorem \ref{thm:0.4}, this definition does not depend on the choice of $B'$.
Thus $\EGK(B)$ depends only on the isomorphism class of $B$.

We will show that $\EGK(B)$ is in fact an EGK datum.
If $F$ is non-dyadic, this follows from Proposition \ref{prop:6.1} and Proposition \ref{prop:6.2}, since $B$ has a Jordan splitting.
To treat the dyadic case, we need some lemmas.
We assume that $F$ is a dyadic field in Lemma \ref{lem:6.1}--\ref{lem:6.4}.

\begin{lemma} %Lemma 6.1
\label{lem:6.1}
Let $B =(b_{ij}) \in \calh_n(\frko)$ be a reduced  form such that $\GK(B)=\ua=(a_1,\ldots,a_n)$.
Assume that $B^{(n-1)}$ is a reduced form with $\GK(B^{(n-1)})=\ua^{(n-1)}$.
\begin{itemize}
\item[(1)] 
Assume that both $n$ and $a_1+\cdots+a_n$ are even. 
Then we have $\eta_B=\eta_{B^{(n-1)}}\xi_B^{a_n}.$
\item[(2)]
Assume that  $n$ is odd and $a_1+\cdots + a_{n-1}$ is even. 
Then we have $\eta_B=\eta_{B^{(n-1)}}\xi_{B^{(n-1)}}^{a_n}.$
\end{itemize}
\end{lemma}
\begin{proof} 
We prove (1).
Let $n$, $\ua$, and $B$ be as in (1).
Then we have
\[
\eta_B=\eta_{B^{(n-1)}} \langle D_B, D_{B^{(n-1)}}\rangle=\eta_{B^{(n-1)}} \xi_B^{a_1+\cdots+a_{n-1}}=\eta_{B^{(n-1)}} \xi_B^{a_n}
\]
by Lemma \ref{lem:3.3}.
Hence we have proved (1).
Similarly, if $n$, $\ua$, and $B$ are as in (2), then we have
\[
\eta_B=\eta_{B^{(n-1)}} \langle D_B, D_{B^{(n-1)}}\rangle=\eta_{B^{(n-1)}} \xi_{B^{(n-1)}}^{a_1+\cdots+a_n}=\eta_{B^{(n-1)}} \xi_{B^{(n-1)}}^{a_n}.
\]
Hence we have proved (2).
\end{proof}

Let $B =(b_{ij}) \in \calh_n(\frko)$ be a reduced  form of standard GK type $(\ua, \sig)$, where $\ua=(a_1,\ldots,a_n)$.
If $n\in \calp^+\cup\calp^0$, then $B^{(n-1)}$ is a reduced form with $\GK(B)=\ua^{(n-1)}$.
Note that if $n\in\calp^+$, then $a_{n-1}<a_n$.

Assume that $a_{n-1}=a_n$ and $a_{\sig(n)}=a_{n-1}$.
Since $\sig$ is standard, we have $\sig(n)=n-1$.
In this case, $B^{(n-2)}$ is a reduced form of GK type $(\ua^{(n-2)}, \sig^{(n-2)})$.
\begin{lemma} % Lemma 6.2 
\label{lem:6.2}
Let $B =(b_{ij}) \in \calh_n(\frko)$ be a reduced  form of standard GK type $(\ua, \sig)$.
Assume that $a_{n-1}=a_n$ and $\sig(n)=n-1$.
\begin{itemize}
\item[(1)] 
Assume that both $n$ and $a_1+\cdots+a_n$ are even.
Then we have $\eta_B=\eta_{B^{(n-2)}}\xi_{B}^{a_n}\xi_{B^{(n-2)}}^{a_n}$.
\item[(2)] 
Assume that $n$ be odd and that $a_1+\cdots +a_{n-1}$ is even.  
Then we have $\eta_{B}=\eta_{B^{(n-2)}}$.
\end{itemize}
\end{lemma}
\begin{proof} %Proof of Lemma 6.2
Write $B$ in a block form
\[
B=
\begin{pmatrix}
B^{(n-2)} & {}^t\!X \\
X & \vpi^{a_n}K
\end{pmatrix},
\]
where $K$ is a primitive unramified binary form.
Put 
\[
B'=B\left[
\begin{pmatrix}
1 & 0 \\
-\vpi^{-a_n}K^{-1} X & 1
\end{pmatrix}
\right]
=\begin{pmatrix}
{B'}^{(n-2)} & 0 \\
0 & \vpi^{a_n}K
\end{pmatrix},
\]
Then we have $B^{(n-2)}-{B'}^{(n-2)}\in \calm^0(\ua^{(n-2)})$.

Suppose that both $n$ and $a_1+\cdots+a_n$ are even.
Note that $\ord(D_{B^{(n-2)}})$ is even and $\xi_{B^{(n-2)}}\neq 0$ in this case.
It follows that $\eta_{B^{(n-2)}}=\eta_{{B'}^{(n-2)}}$ by Proposition \ref{prop:3.3} (c).
Then, by Lemma \ref{lem:3.4}, we have 
\[
\eta_B=\eta_{B^{(n-2)}} \xi_K^{a_n}=\eta_{B^{(n-2)}}\xi_{B}^{a_n}\xi_{B^{(n-2)}}^{a_n}.
\]
Now suppose that $n$ is odd and $a_1+\cdots+a_{n-1}$ is even.
In this case, $\eta_{B^{(n-2)}}=\eta_{{B'}^{(n-2)}}$ by Proposition \ref{prop:3.3} (b).
Note that $\ord(D_{B^{(n-2)}})=a_1+\cdots+a_{n-2}$.
Then, by Lemma \ref{lem:3.4}, we have
\[
\eta_B=\eta_{B^{(n-2)}} \xi_K^{a_n+(a_1+\cdots+a_{n-2})}=\eta_{B^{(n-2)}}.
\]
Hence the lemma.
\end{proof}

\begin{lemma} % Lemma 6.3
\label{lem:6.3} 
For a non-decreasing sequence $\ua=(a_1,\ldots,a_n)$ of integers, let $\sig$ be an $\ua$-admissible involution. 
Let $B \in \calh_n(\frko)$ be a reduced form of standard GK type $(\ua,\sig)$. 
Assume that $r\geq 2$.
Put $k=n^\ast_{r-1}$.
\begin{itemize}
\item[(1)] 
Assume that both $n$ and $k$ are odd and that $a_1+\cdots+a_{n-1}$ is even. 
Then $\eta_B=\eta_{B^{(k)}}$.
\item[(2)] 
Assume that $n$ is odd, $k$ is even, and that $a_1+\cdots+a_k$ is even. 
Then $\eta_B=\eta_{B^{(k)}} \xi_{B^{(k)}}^{a_n}$.
\item[(3)] 
Assume that both $n$ and $k$ are even and that $a_1+\cdots+a_n$ is even. 
Then $\eta_B=\eta_{B^{(k)}} \xi_{B^{(k)}}^{a_n} \xi_B^{a_n}$.
\item[(4)] 
Assume that $n$ is even, $k$ is odd, and that $a_1+\cdots+a_n$ is even. 
Then $\eta_B=\eta_{B^{(k)}} \xi_{B}^{a_n}$.
\end{itemize}
\end{lemma}
\begin{proof} 
We proceed by induction with respect to $n_r$.
If $n_r=1$, then (2) and (4) follow from Lemma \ref{lem:6.1} (2) and Lemma \ref{lem:6.1} (1), respectively.
If $n_r=2$, then (1) and (3)  follow from Lemma \ref{lem:6.2} (2) and Lemma \ref{lem:6.2} (1), respectively.
If $n_r\geq 3$, the lemma follows by using Lemma \ref{lem:6.2} repeatedly.
\end{proof}

\begin{lemma} % Lemma 6.4
\label{lem:6.4}
Let $B \in \calh_n(\frko)$ be a reduced form of GK type $(\ua,\sig)$. 
Put
\[
\EGK(B)= (n_1,\ldots,n_r; m_1,\ldots,m_r;\zeta_1,\ldots,\zeta_r).
\]
\begin{itemize}
\item[(1)] 
Assume that $n^\ast_i$ is even for $i=1, 2, \ldots, r$. 
Then we have
\[
\eta_B=
\xi_1^{m_1+m_2}\cdots \xi_{r-1}^{m_{r-1}+m_r} \cdot \xi_r^{m_r}.
\]
\item[(2)]
Assume that $n$ is even and that $t< r$ is the largest number such that $n^\ast_t$ is odd.
Assume also that $m_1n_1+\cdots+m_r n_r$ is even.
Then we have
\[
\eta_B=\zeta_t
\xi_{t+1}^{m_{t+1}+m_{t+2}}\cdots \xi_{r-1}^{m_{r-1}+m_r} \xi_{r-1}^{m_r}.
\]
\end{itemize}
\end{lemma}
\begin{proof}
Let $n$, $B$, and $(\ua, \sig)$ be as in (1).
If $r=1$, then (1) is a special case of  Lemma \ref{lem:3.5}.
For $r>1$, (1) can be proved by applying Lemma \ref{lem:6.3} (3), repeatedly.

Next, we shall prove (2).
Let $n$, $B$, and $(\ua, \sig)$ be as in (2).
Note that $m_1 n_1 +\cdots+m_{t+1}n_{t+1}$ is even.
Then we have $\eta_{B^{(n^\ast_{t+1})}}=\eta_t\xi_{t+1}^{a_{t+1}}$ by applying Lemma \ref{lem:6.3} (4).
By using Lemma \ref{lem:6.3} (3) repeatedly, we have (2).
\end{proof}

\begin{theorem} % Theorem 6.1
\label{thm:6.1}
Let $F$ be a non-archimedean local field.
Suppose that $B \in \calhnd_n(\frko )$. 
Then $\EGK(B)$ is an $\EGK$ datum of length $n$.
\end{theorem}
\begin{proof} 
It is enough to consider the case when $F$ is dyadic.
We may assume $B$ is  a reduced form of standard GK type $(\ua, \sig)$. 
Put
\[
\EGK(B)= (n_1,\ldots,n_r;m_1,\ldots,m_r;\zeta_1,\ldots,\zeta_r).
\]
The condition (E1) is obvious, and the condition (E2) follows from Theorem \ref{thm:0.1}.

We will prove the condition (E3) holds.
Suppose that $n^\ast_s$ is odd.
Replacing $B$ by $B^{(s)}$, we may assume $s=r$.
It is obvious that $\zeta_r=\eta_B\neq 0$.

Assume that $n^\ast_i$ is even for any $i<r$. 
By Lemma \ref{lem:6.4} (1), we have
\[
\eta_{B^{(n^\ast_{r-1})}}=
\xi_1^{m_1+m_2}\cdots \xi_{r-2}^{m_{r-2}+m_{r-1}} \cdot \xi_{r-1}^{m_{r-1}}.
\]
Then by Lemma \ref{lem:6.3} (2), we have
\[
\zeta_r=\eta_B=
\xi_1^{m_1+m_2}\cdots \xi_{r-1}^{m_{r-1}+m_r}. 
\]
Hence (a) of the condition (E3) holds.

Next, assume $m_1n_1+\cdots+m_{r-1}n_{r-1}+m_r (n_r-1)$ is even and that $n^\ast_i$ is odd for some $i<r$.
Let $t<r$ be the largest number such that $n^\ast_t$ is odd.
If $r=t+1$, then we have $\eta_B=\eta_{B^{(n^\ast_t)}}$ by Lemma \ref{lem:6.3} (1).
Hence (b) of the condition (E3) holds in this case.
Now, assume that $r>t+1$.
By Lemma \ref{lem:6.4} (2), we have
\[
\eta_{B^{(n^\ast_{r-1})}}=\zeta_t
\xi_{t+1}^{m_{t+1}+m_{t+2}}\cdots \xi_{r-2}^{m_{r-2}+m_{r-1}} \xi_{r-1}^{m_{r-1}}.
\]
By Lemma \ref{lem:6.3} (2), we have
\[
\zeta_r=\eta_B=
\zeta_t
\xi_{t+1}^{m_{t+1}+m_{t+2}}\cdots \xi_{r-1}^{m_{r-1}+m_r}.
\]
Hence, $\EGK(B)$ satisfies (b) of the condition (E3).
\end{proof}
We call $\EGK(B)$ the extended GK datum associate to $B$.
One can prove the following proposition, but as we do not use it later, we omit a proof.
\begin{proposition} % Proposition 6.4
\label{prop:6.4}
Suppose that $F$ is a dyadic local field.
Let $G=(n_1,\ldots,n_r;m_1,\ldots,m_r;\zeta_1,\ldots,\zeta_r)$ be an EGK datum and $(\ua, \sig)$ a standard GK type such that $(n_1,\ldots,n_r;m_1,\ldots,m_r)=\Ups(\ua)$.
Then there exists  a reduced form $B\in\calhnd_n(\frko)$ of GK type $(\ua, \sig)$ such that $\EGK(B)=G$.
\end{proposition}

%%%%%%%%%%%%%%%%%%%%%%%%%%%%%%%%%%%%%%%%%%%%

\end{document}